\documentclass[12pt,reqno]{amsproc}
\usepackage{times}

\usepackage{amsmath,amsthm,amssymb,wasysym,mathtools}
\usepackage{float}

\usepackage[inline]{enumitem}
\usepackage{caption}
\usepackage{subcaption}
\usepackage{graphicx}

\DeclareMathOperator{\RE}{Re}

\numberwithin{equation}{section}

\usepackage{amsthm}
\usepackage{tikz}

\usepackage{pgfplots}
\pgfplotsset{compat=1.18} 

\newtheoremstyle{fancytheorem}               
{.5\baselineskip±.2\baselineskip}           
{.5\baselineskip±.2\baselineskip}           
{\itshape\addtolength{\leftskip}{0mm}\setlength{\parindent}{0em}}      
{0mm}                                     
{\bfseries}                                 
{}                                  
{.5em}                                      
{\thmname{#1}\thmnumber{ #2}. \thmnote{[#3]}}

\pgfmathdeclarefunction{mysign}{1}{%
  \pgfmathparse{(#1>0)-(#1<0)}%
}
\pgfmathdeclarefunction{arccoshf}{1}{%
  \pgfmathparse{ln(#1+sqrt((#1)^2-1))}%
}
\pgfmathdeclarefunction{arcsinrad}{1}{%
  \pgfmathparse{asin(#1)*pi/180}%
}

\pgfmathdeclarefunction{qrf}{1}{%
  \pgfmathparse{1 + mysign(cos(#1*180/pi)) * arccoshf(sqrt(1+abs(cos(#1*180/pi))))}%
}

\pgfmathdeclarefunction{qif}{1}{%
  \pgfmathparse{mysign(sin(#1*180/pi)) * arcsinrad(sqrt(1-abs(cos(#1*180/pi))))}%
}
\pgfmathdeclarefunction{den}{1}{%
  \pgfmathparse{(4 + 2*sqrt(2)) - (2*sqrt(2) + 2)*cos(#1)}%
}
\pgfmathdeclarefunction{expx}{1}{%
  \pgfmathparse{exp(cos(#1))}%
}
\pgfmathdeclarefunction{ex}{1}{%
  \pgfmathparse{exp(-cos(#1))}%
}

\pgfmathdeclarefunction{den1}{1}{%
  \pgfmathparse{1 + 2*ex(#1)*cos(deg(sin(#1))) + (ex(#1))^2}%
}

\newtheoremstyle{fancydefinition}               
{.5\baselineskip±.2\baselineskip}           
{.5\baselineskip±.2\baselineskip}           
{\upshape\addtolength{\leftskip}{0mm}\setlength{\parindent}{0em}}      
{0mm}                                     
{\bfseries}                                 
{}                                  
{.5em}                                      
{\thmname{#1}\thmnumber{ #2}. \thmnote{[#3]}}

\theoremstyle{fancytheorem}

\newtheorem{theorem}{Theorem}[section]
\newtheorem{lemma}[theorem]{Lemma}
\newtheorem{corollary}[theorem]{Corollary}
\newtheorem{remark}[theorem]{Remark}

\theoremstyle{fancydefinition}

\allowdisplaybreaks
\usepackage{xcolor}

\newcommand{\D}{\mathbb{D}}
\newcommand{\A}{\mathcal{A}}
\newcommand{\Q}{\mathcal{Q}_{n}(M)}

\allowdisplaybreaks

\begin{document}

\title[Radius of Ma--Minda Starlikeness and Convexity]{Radius of Ma--Minda Starlikeness and Convexity\\ for Functions with Bounded Second Derivative}

\author[P. Poomalai]{Palanisamy Poomalai }
\address{Department of Mathematics \\National Institute of Technology\\Tiruchirappalli-620015,  India }
\email{poomalainitt@gmail.com}

\author[V. Ravichandran]{Vaithiyanathan Ravichandran }
\address{Department of Mathematics \\National Institute of Technology\\Tiruchirappalli-620015,  India }
\email{vravi68@gmail.com; ravic@nitt.edu}

\date{}

\begin{abstract}A normalized analytic function $f(z)=z+\sum_{k=n+1}^{\infty}a_kz^k$ defined on the unit disk $\mathbb{D}$ is called Ma--Minda starlike or Ma--Minda convex if $zf'(z)/f(z)$ or $1+zf''(z)/f'(z)$, respectively, is subordinate to a Ma--Minda function $\varphi$. In this paper, we investigate the radii of Ma--Minda starlikeness and convexity for functions satisfying $|f''(z)|\le M$ in $\mathbb{D}$. We first develop a unified approach for determining these radii and then apply it to obtain explicit radius estimates for several important choices of the Ma--Minda function~$\varphi$.

\end{abstract}



\maketitle

\section{Introduction}

Let  $\D=\{z\in\mathbb C:|z|<1\}$ denote the open unit disk in the complex plane. Let $\A$ be the class of analytic functions $f:\D\to\mathbb C$ normalized by
$f(0)=0$ and $ f'(0)=1$.  For each integer $n\ge1$, let $\A_n$ denote the subclass of $\A$ consisting of functions of the form
$f(z)=z+\sum_{k=n+1}^{\infty}a_kz^k$. Clearly, $ \A=\A_1\supset\A_2\supset\cdots\supset\A_n\supset\cdots.
$
Following \cite{mainpaper}, let $\Q$ denote the class of functions $f\in\A_n$ satisfying
$|f''(z)|\le M$ for all $z\in\D$,  where $M>0$ is a prescribed constant.
Among the most studied subclasses of the class $\mathcal S$ of univalent functions are the classes of starlike and convex functions. A function $f\in\A$ is said to be \emph{starlike} if the image domain $f(\D)$ is starlike with respect to the origin. Likewise, $f$ is said to be \emph{convex} if $f(\D)$ is a convex domain. Analytically,
\[
f\in \mathcal{ST}
\quad\Longleftrightarrow\quad
\RE\left(\frac{zf'(z)}{f(z)}\right)>0,
\]and\[
f\in\mathcal{CV}
\quad\Longleftrightarrow\quad
\RE\left(1+\frac{zf''(z)}{f'(z)}\right)>0.
\]
The above characterizations naturally led to  generalizations through the theory of differential subordination. Let $p$ and $q$ be analytic in $\D$. We say that $p$ is \emph{subordinate} to $q$, written $p\prec q$, if there exists an analytic function $w:\D\to\D$ with $w(0)=0$ such that $p(z)=q(w(z))$. If, in addition, $q$ is univalent in $\D$, then $p\prec q$ is equivalent to the conditions $p(0)=q(0)$ and $p(\D)\subseteq q(\D)$. Ma and Minda introduced a unified family of subclasses of starlike and convex functions associated with an analytic function $\varphi$ having positive real part in $\D$, satisfying $\varphi(0)=1$, $\varphi'(0)>0$, and mapping $\D$ conformally onto a domain that is starlike with respect to the point $1$ and symmetric about the real axis. The corresponding Ma--Minda classes are defined by
\[
\mathcal{ST}(\varphi)
=
\left\{
f\in\A:
\frac{zf'(z)}{f(z)}
\prec
\varphi(z)
\right\},\] and \[
\mathcal{CV}(\varphi)
=
\left\{
f\in\A:
1+\frac{zf''(z)}{f'(z)}
\prec
\varphi(z)
\right\}.
\]
Suitable choices of $\varphi$ recover many well-known subclasses of starlike and convex functions that have been extensively studied in geometric function theory.

The primary objective of this paper is to determine the Ma--Minda radii of starlikeness and convexity for the class $\Q$. Our approach is based on sharp disk estimates for
$zf'(z)/f(z)$ and $1+ zf''(z)/f'(z)$, together with   disk inclusion results for the corresponding Ma--Minda domains. This leads to a general radius theorem expressed in terms of the largest   disk centered at $1$ and contained in the image domain.
For all the Ma--Minda domains considered, with the exception of the Bernoulli lemniscate domain, the largest such disk is of the form $\mathbb{D}(1;1-\delta)$, where $\delta=\varphi(-1)$ is the leftmost point at which the boundary of the domain intersects the real axis. The Bernoulli lemniscate domain is the only exception, for which the largestdisk centered at $1$ has radius $\sqrt{2}-1$.
As applications of the general theorem, we obtain explicit radius formulae for fourteen Ma--Minda subclasses. The results are shown to be sharp for thirteen of these subclasses by constructing suitable extremal functions, while an explicit lower bound is established for the Bernoulli lemniscate subclass. Finally, for the case $n=M=1$, we provide figures illustrating all fourteen Ma--Minda domains together with the corresponding extremal images that establish the sharp radius results and the lower-bound result for the Bernoulli lemniscate subclass.

\section{Radius of starlikeness and convexity}

The following estimates, established in \cite[Corollary 1]{mainpaper}, form the basis for our investigation.

\begin{lemma}\label{mainlemma}
If $f\in\Q$, then, for  $z\in\D$,
\begin{align*}
\left|
\frac{zf'(z)}{f(z)}-1
\right| &
\le
\frac{nM|z|^n}{n(n+1)-M|z|^n},
\qquad  0<M\le n(n+1),\\
\intertext{and }
\left|
\frac{zf''(z)}{f'(z)}
\right| &
\le
\frac{nM|z|^n}{n-M|z|^n},
\qquad 0<M\le n.
\end{align*}
\end{lemma}

We now establish a general theorem that yields the Ma--Minda radii of starlikeness and convexity for the class $\Q$. All subsequent radius results follow as direct applications of this theorem together with the appropriate disk inclusion lemma.

	\begin{theorem}\label{main theorem}
	Let
	$
	\mathbb D(1;1-\delta)\subseteq \varphi(\mathbb{D})
	$,
	for some $0<\delta<1$. Let \[
	R_1=
	\min\left\{
	1,
	\left(
	\frac{(1-\delta)n(n+1)}
	{M(n+1-\delta)}
	\right)^{1/n}
	\right\},\] and \[
	R_2=	\min\left\{
	1,
	\left(
	\frac{(1-\delta)n}
	{M(n+1-\delta)}
	\right)^{1/n}
	\right\}.
	\]
	The Ma--Minda starlikeness radius $R_{\mathcal{ST}(\varphi)}(\Q)$ of $\Q$ is at least $R_1$ whenever $0<M\le n(n+1)$, whereas the Ma--Minda convexity radius $R_{\mathcal{CV}(\varphi)}(\Q)$ is at least $R_2$ whenever $0<M\le n$.
	\end{theorem}

\begin{proof}
Let $f\in\Q$ with $0<M\le n(n+1)$. By Lemma~\ref{mainlemma},
$zf'(z)/f(z) \in \D(1;c_1(r))$,  where \[
c_1(r)=\frac{nMr^n}{n(n+1)-Mr^n}.
\label{eq:2.3}
\]
Since the center of the disk estimate is fixed at $1$ and
$
\D(1;1-\delta)\subseteq\varphi(\mathbb{D}),
$
it suffices to determine the largest radius $r$ for which
$
\D(1;c_1(r))
\subseteq
\D(1;1-\delta).
$
Let
\[
h(r)
=
1-c_1(r)
=
1-\frac{nMr^n}{n(n+1)-Mr^n},
\]
so that $h(r)$ denotes the leftmost point  on the real axis  of the disk $\D(1;c_1(r))$.
A direct computation gives
\[
h'(r)
=
-\frac{Mn^3(n+1)r^{\,n-1}}
{\left(n(n+1)-Mr^n\right)^2}
<0,
\qquad
0<r<1.
\]
Hence the function $h$ is strictly decreasing on $(0,1)$, meaning that the disk
$\D(1;c_1(r))$ expands and its left endpoint moves monotonically to the left
as $r$ increases.
The largest admissible radius is determined when $h(r)$ reaches the
boundary point $\delta$. Thus we solve
$
h(r)=\delta,
$
or equivalently,
$
c_1(r)=1-\delta.
$
This reduces to
\[
\Phi(r)
:=
Mr^n(n+1-\delta)
-
(1-\delta)n(n+1).
\]
Whether this equation admits a solution in $(0,1)$ depends on the value of
$M$, and hence we consider two cases.

\medskip

\noindent\textbf{Case (i).}
Suppose that
\[
\frac{(1-\delta)n(n+1)}
{n+1-\delta}
<
M
\le
n(n+1).
\]
Then $
\Phi(0)
=
-(1-\delta)n(n+1)
<0$,
whereas
$
\Phi(1)
=
M(n+1-\delta)
-
(1-\delta)n(n+1)
>0$.
Since $\Phi$ is continuous and strictly increasing on $(0,1)$, it has a unique
zero in $(0,1)$, namely
\[
R_1
=
\left(
\frac{(1-\delta)n(n+1)}
{M(n+1-\delta)}
\right)^{1/n}.
\]

\medskip

\noindent\textbf{Case (ii).}
Suppose that
\[
0<M
\le
\frac{(1-\delta)n(n+1)}
{n+1-\delta}.
\]
Then
$
\Phi(1)=M(n+1-\delta)-(1-\delta)n(n+1)\le0$.
Since $\Phi$ is increasing on $(0,1)$,
$
\Phi(r)\le0$, for $0<r<1$.
Equivalently,
\[
c_1(r)
=
\frac{nMr^n}
{n(n+1)-Mr^n}
\le
1-\delta.
\]
Hence
$\D(1;c_1(r))\subseteq \D(1;1-\delta)$, for $0<r<1$,
and therefore
$
R_1=1$.
 Combining the two cases,
\[
R_1
=
\min\left\{
1,
\left(
\frac{(1-\delta)n(n+1)}
{M(n+1-\delta)}
\right)^{1/n}
\right\}.
\]
Finally, since $h$ is decreasing,
$
h(r)\ge h(R_1)=\delta$,
for all
$0<r\le R_1.
$
Hence
$
c_1(r)\le1-\delta,
$
so that
$
\D(1;c_1(r))
\subseteq
\D(1;1-\delta)
\subseteq
\varphi(\mathbb{D})$ for $0\le r\le R_1.
$
Therefore,
$
zf'(z)/f(z)
\in
\varphi(\mathbb{D}),
\text{for} \,
|z|\le R_1,
$
which proves
$
R_{\mathcal {ST}(\varphi)}(\Q)
\ge
R_1$.
	
If $f\in\Q$ with $0<M\le n$, by Lemma~\ref{mainlemma}, we have
	$	1+ zf''(z)/f'(z)	\in	\D(1;c_2(r))$,	where
	\[
	c_2(r)=\frac{nMr^n}{n-Mr^n}.
	\label{eq:1.4}
	\]
	Since 	$
	\D(1;1-\delta)\subseteq\varphi(\mathbb{D}),
	$
	it suffices to determine the largest radius $r$ for which
	$
	\D(1;c_2(r))
	\subseteq
	\D(1;1-\delta)$.
	Let
	\[
	g(r)
	=
	1-c_2(r)
	=
	1-\frac{nMr^n}{n-Mr^n}.
	\]A direct computation gives
	\[
	g'(r)
	=
	-\frac{Mn^3r^{\,n-1}}
	{\left(n-Mr^n\right)^2}
	<0,
	\qquad
	0<r<1.
	\]
	Hence the function $g$ is strictly decreasing on $(0,1)$, meaning that the disk
	$\D(1;c_2(r))$ expands and its left endpoint moves monotonically to the left
	as $r$ increases.   The largest admissible radius is determined when $g(r)$ reaches the
	boundary point $\delta$. Thus we solve
	$
	g(r)=\delta,
	$
	or equivalently,
	$
	c_2(r)=1-\delta.
	$
	This reduces to
	\[
	\Psi(r)
	:=
	Mr^n(n+1-\delta)
	-
	(1-\delta)n.
	\]
	The location of the solution depends on $M$, leading to two cases.
	
	\medskip
	
	\noindent\textbf{Case (i).}
	Suppose that
	\[
	\frac{(1-\delta)n}
	{n+1-\delta}
	<
	M
	\le
	n.
	\]
	Then
	$\Psi(0)
	=
	-(1-\delta)n
	<0$,
	whereas $
	\Psi(1)
	=
	M(n+1-\delta)
	-
	(1-\delta)n
	>0$.
	Since $\Psi$ is continuous and strictly increasing on $(0,1)$, it has a unique
	zero in $(0,1)$, namely
	\[
	R_2
	=
	\left(
	\frac{(1-\delta)n}
	{M(n+1-\delta)}
	\right)^{1/n}.
	\]
	
	\medskip
	
	\noindent\textbf{Case (ii).}
	Suppose that
	\[
	0<M
	\le
	\frac{(1-\delta)n}
	{n+1-\delta}.
	\]
	Then
	\[
	\Psi(1)
	=
	M(n+1-\delta)
	-
	(1-\delta)n
	\le0.
	\]
	Since $\Psi$ is increasing on $(0,1)$,
	$
	\Psi(r)\le0,$ for $
	0<r<1.
	$
	Equivalently,
	\[
	c_2(r)
	=
	\frac{nMr^n}
	{n-Mr^n}
	\le
	1-\delta.
	\]
	Hence
	$
	\D(1;c_2(r))
	\subseteq
	\D(1;1-\delta)$, for all $
	0<r<1,
	$
	and therefore
	$
	R_2=1$.  Combining the two cases,
	\[
	R_2
	=
	\min\left\{
	1,
	\left(
	\frac{(1-\delta)n}
	{M(n+1-\delta)}
	\right)^{1/n}
	\right\}.
	\]
	Finally, since $g$ is decreasing,
	$
	g(r)\ge g(R_2)=\delta,$
	for all
	$0<r\le R_2.
	$
	Hence
	$
	c_2(r)\le1-\delta,
	$
	so that
	$
	\D(1;c_2(r))
	\subseteq
	\D(1;1-\delta)
	\subseteq
	\varphi(\mathbb{D}),
	\,
	0\le r\le R_2.
	$
	Therefore,
	$
	1+zf''(z)/f'(z)
	\in
	\varphi(\mathbb{D})$ for  $|z|\le R_2,
	$
	which proves
$	R_{\mathcal {CV}(\varphi)}(\Q)
	\ge
	R_2$.
	\end{proof}
	
	The exponential starlike and convex classes were introduced by Mendiratta \textit{et al.}~\cite{MR3394060} in 2015 as special cases of the Ma--Minda classes corresponding to the function
	$\varphi_e(z)=e^z$.  Specifically, $\mathcal{ST}_e=\mathcal{ST}(\varphi_e)$ and $ \mathcal{CV}_e=\mathcal{CV}(\varphi_e)$.
	Equivalently, a function $f\in\mathcal{A}$ belongs to $\mathcal{ST}_e$ if
	$
	\left|\log\!\left({zf'(z)}/{f(z)}\right)\right|<1,
	$
	and belongs to $\mathcal{CV}_e$ if
	$
	\left|\log\!\left(1+{zf''(z)}/{f'(z)}\right)\right|<1.$
	
		\begin{lemma}\label{lemma1}\cite[Lemma 2.2]{MR3394060}
			For $1/e<a<e,\ \{w:|w-a|<r_a\}\subseteq \varphi_{e}(\mathbb{D})$, where $r_a$ is
			\begin{align*}
				r_a &=
				\begin{dcases}
					a-e^{-1},\ & \ e^{-1}<a\leq \frac{e+e^{-1}}{2}\\
					e-a,\ & \frac{e+e^{-1}}{2}\leq a<e.
				\end{dcases}
			\end{align*}
		\end{lemma}
		Using this lemma we determine the radii of starlikeness and convexity associated with exponential  function  for the class $\Q$

	\begin{corollary}\label{corollary1}Let	\[
R_1
	=
	\min\left\{
	1,
	\left(
	\frac{\left(1-e^{-1}\right)n(n+1)}
	{M\left(n+1-e^{-1}\right)}
	\right)^{1/n}
	\right\},\] and \[
R_2
	=
	\min\left\{
	1,
	\left(
	\frac{\left(1-e^{-1}\right)n}
	{M\left(n+1-e^{-1}\right)}
	\right)^{1/n}
	\right\}.
	\]
	The $\mathcal {ST}_{e}$-radius of the class $\Q$ is $R_1$ and the $\mathcal {CV}_{e}$-radius    is $R_2$.
	\end{corollary}
	\begin{proof}
	By Lemma~\ref{lemma1},
	$
	\mathbb D\!\left(1;1-1/e\right)\subseteq\varphi_e(\mathbb{D}).
	$
	Hence, applying Theorem~\ref{main theorem} with
	\(
	\delta=1/e,
	\)
	we obtain the stated lower bounds for the
	$\mathcal {ST}_e$-
	and
	$\mathcal {CV}_e$-radii.
	To prove sharpness, consider the extremal function $f_0(z):\D \rightarrow \mathbb C$ defined by
	\begin{equation}\label{eq:sharp-stt}
	f_0(z)
	=
	z+\frac{M}{n(n+1)}z^{n+1}.
	\end{equation}
	Then
	\begin{equation}\label{eq:sharp-st}
	\frac{zf_0'(z)}{f_0(z)}
	=
	1+\frac{nMz^n}
	{n(n+1)+Mz^n},
	\end{equation}
	and
	\begin{equation}\label{eq:sharp-cv}
	1+\frac{zf_0''(z)}{f_0'(z)}
	=
	1+\frac{nMz^n}
	{n+Mz^n}.
	\end{equation}
	Suppose that
	\[
	\rho=
	\left(
	\frac{(1-e^{-1})n(n+1)}
	{M(n+1-e^{-1})}
	\right)^{1/n}<1,
	\]
	Let $z_0 \in \D$ be chosen so that $|z_0|=\rho$
	and
	$z_0^n=-\rho^n$.
	By the definition of $\rho$,
	\[
	M\rho^n\left(n+1-\frac1e\right)
	=
	\left(1-\frac1e\right)n(n+1) \implies \frac{nM\rho^n}
		{n(n+1)-M\rho^n}=1-\frac1e
	\]
	By \eqref{eq:sharp-st}, evaluated at $z=z_0$, we have
	\[
	\frac{z_0f_0'(z_0)}{f_0(z_0)}
	=
	1-\frac{nM\rho^n}
	{n(n+1)-M\rho^n}=\frac1e=\varphi_e(-1)\in\partial \varphi_e(\mathbb{D})
	\]
	Hence the
	$\mathcal {ST}_e$-radius is sharp.
	
	Similarly, suppose that
	\[
	\rho=
	\left(
	\frac{(1-e^{-1})n}
	{M(n+1-e^{-1})}
	\right)^{1/n}<1.
	\]
	Then
	\[
	M\rho^n\left(n+1-\frac1e\right)
	=
	\left(1-\frac1e\right)n \implies \frac{nM\rho^n}
		{n-M\rho^n}
		= 1-
		\frac1e
	\]
	and therefore,
	using \eqref{eq:sharp-cv}, evaluated at $z=z_0$, we have
	\[
	1+\frac{z_0f_0''(z_0)}{f_0'(z_0)}
	=
	1-\frac{nM\rho^n}
	{n-M\rho^n}
	=
	\frac1e
	=
	\varphi_e(-1)
	\in
	\partial\varphi_e(\mathbb{D}).
	\]
	Thus   the
	$\mathcal {CV}_e$-radius is also sharp.
	\end{proof}

	\begin{remark}
	For the particular class $\mathcal{Q}_1(1)$, corresponding to $n=1$ and $M=1$, the
	radius formulas in Corollary~\ref{corollary1} reduce to
	\[
	R_{\mathcal {ST}_e}(\mathcal{Q}_1(1))
	=
	\frac{2( e-1)}
	{2e-1 }
	\approx
	0.7746,
	\quad \text{	and }\quad
	R_{\mathcal{CV}_e}(\mathcal{Q}_1(1))
	=
	\frac{e-1}
	{2e-1}
	\approx
	0.3873.
	\]
	\end{remark}

	
	The subclasses of starlike functions associated with the cardioid were introduced by Sharma \textit{et al.}~\cite{MR3536076} as Ma--Minda classes corresponding to the function
	$
	\varphi_C(z)=1+{4z}/{3}+{2z^2}/{3}.
	$
	Specifically,
	$\mathcal{ST}_C=\mathcal{ST}(\varphi_C)$ and $\mathcal{CV}_C=\mathcal{CV}(\varphi_C)$.
	
	\begin{lemma}\label{lemma2}\cite[Lemma 2.5]{MR3536076}
			Let $ \varphi(\mathbb{D})=\varphi_C(\D)$ is the region bounded by cardioid $ \{x+iy:(9x^2+9y^2-18x+5)^2-16(9x^2+9y^2-6x+1)=0\} $. Then for $1/3<a<3$, $\{w:|w-a|<r_a\}\subset\varphi_C(\mathbb{D})$, where $r_a$ is given by
			\[
			r_a=
			\begin{cases}
			\dfrac{3a-1}{3}, & \dfrac{1}{3}<a\le\dfrac{5}{3}\\[2.5mm]
			3-a, & \dfrac{5}{3}\le a<3.
			\end{cases}
			\]
			
		\end{lemma}
	
	\begin{corollary}\label{corollary2}Let
	\[
R_1
	=
	\min\left\{
	1,
	\left(
	\frac{2n(n+1)}
	{M\left(3n+2\right)}
	\right)^{1/n}
	\right\},
	\quad \text{and}\quad
R_2
	=
	\min\left\{
	1,
	\left(
	\frac{2n}
	{M\left(3n+2\right)}
	\right)^{1/n}
	\right\}.
	\]
	The $\mathcal {ST}_{C}$-radius of the class $\Q$ is $R_1$ and the $\mathcal {CV}_{C}$-radius  is $R_2$.
	\end{corollary}
	
	\begin{proof}
		By Lemma~\ref{lemma2},
		$
		\mathbb D\!\left(1;1-1/3\right)\subset\varphi_C(\mathbb{D}).
		$
		Hence, Theorem~\ref{main theorem} with
		\(
		\delta=1/3
		\)
		immediately yields the asserted lower bounds for the
		$\mathcal {ST}_{C}$- and $\mathcal {CV}_{C}$-radii. To prove sharpness, let \[
	\rho
	=
	\left(
	\frac{2n(n+1)}
	{M(3n+2)}
	\right)^{1/n}
	<1
	\] and  consider the extremal function $f_0:\D \rightarrow \mathbb C$ defined by \eqref{eq:sharp-stt}.
Choose $z_0 \in \D$ with $|z_0|=\rho$ such that $z_0^n=-\rho^n$.
	By the definition of $\rho$, it follows that
	\[
	M\rho^n\left(n+\frac23\right)
	=
	\frac23n(n+1) \implies
	\frac{nM\rho^n}
	{n(n+1)-M\rho^n}
	=
	1-\frac13.
	\]
	Hence, by \eqref{eq:sharp-st}, evaluated at $z=z_0$, we have
	\[
	\frac{z_0f_0'(z_0)}{f_0(z_0)}
	=1-\frac{nM\rho^n}{n(n+1)-M\rho^n}
	=\frac13
	=\varphi_C(-1)
	\in
	\partial\varphi_C(\mathbb{D}),
	\]
	which proves sharpness of the $\mathcal {ST}_C$-radius.
	
	Similarly, let
	\[
	\rho
	=
	\left(
	\frac{2n}
	{M(3n+2)}
	\right)^{1/n}
	<1.
	\]
	By the definition of $\rho$, it follows that
	\[
	M\rho^n\left(n+\frac23\right)
	=
	\frac23n
	\implies	
	\frac{nM\rho^n}
	{n-M\rho^n}
	=
	1-\frac13.
	\]
	Using \eqref{eq:sharp-cv},  evaluated at $z=z_0$, we have
	\[
	1+\frac{z_0f_0''(z_0)}{f_0'(z_0)}
	=	1-
		\frac{nM\rho^n}
		{n-M\rho^n}=
	\frac13
	=
	\varphi_C(-1)
	\in
	\partial\varphi_C(\mathbb{D}),
	\]
	this proves the sharpness of the
	$\mathcal {CV}_{C}$-radius.
	\end{proof}
	
	\begin{remark}
	For the particular class $\mathcal{Q}_1(1)$, corresponding to $n=1$ and $M=1$, the
	radius formulas in Corollary~\ref{corollary2} reduce to
	\[
	R_{\mathcal {ST}_C}(\mathcal{Q}_1(1))
	=
	\frac{4}{5}
	=
	0.8,
	\quad \text{and}
\quad
	R_{\mathcal {CV}_C}(\mathcal{Q}_1(1))
	=
	\frac{2}{5}
	=
	0.4.
	\]
	\end{remark}
	
	\begin{figure}[!htbp]
	\centering
	\begin{subfigure}{0.48\textwidth}
	\centering
	
		\begin{tikzpicture}
			\begin{axis}[
			    axis lines=center,
			    xmin=-0.3, xmax=3.0,
			    ymin=-1.4, ymax=1.4,
			    xtick={ 1.0, 2.0},
			    ytick={-1.0, 1.0},
			    thick,
			    scale=0.75
			]
			
			   \addplot[
					domain=0:360,
					samples=400,
					smooth,
					color=black,
					line width=1pt
				]
				({0.8235 + 0.45546*cos(x)},
				{0.45546*sin(x)});
				
			    \addplot[
			        domain=0:360,
			        samples=400,
			        smooth,
			        color=black,
			        line width=1pt
			    ]
			    ({1 + (1-exp(-1))*cos(x)},
			     {(1-exp(-1))*sin(x)});
			
			    \addplot[
			        domain=0:360,
			        samples=250,
			        smooth,
			        color=black,
			        line width=1pt
			    ]
			    ({exp(cos(x))*cos(deg(sin(x)))}, {exp(cos(x))*sin(deg(sin(x)))});
			
			\end{axis}
			\end{tikzpicture}
	\caption{$\mathcal{ST}_e$-Radius}
	\end{subfigure}
	\begin{subfigure}{0.48\textwidth}
	\centering
\begin{tikzpicture}
		\begin{axis}[
		    axis lines=center,
		    width=7.8cm,
		    height=7.5cm,
		    xmin=-0.4, xmax=3.3,
		    ymin=-1.8, ymax=1.8,
		    xtick={1.0,2.0,3.0},
		    ytick={-1.0,1.0},
		    thick,
		    scale=0.75
		]
	   \addplot[
	   			domain=0:360,
	   			samples=400,
	   			smooth,
	   			color=black,
	   			line width=1pt
	   		]
	   		({0.8095 + 0.476*cos(x)},
	   		{0.476*sin(x)});
	
	\addplot[
				domain=0:360,
				samples=400,
				smooth,
			color=black,
				line width=1pt
			]
			({1 + 0.66*cos(x)},
			{0.66*sin(x)});
		
		\addplot[
		    domain=-180:180,
		    samples=250,
		    smooth,
		    color=black,
		    line width=1pt
		]
		({1 + (4/3)*cos(x) + (2/3)*cos(2*x)},
		 {(4/3)*sin(x) + (2/3)*sin(2*x)});
		
		\end{axis}
		\end{tikzpicture}
		\caption{$\mathcal {ST}_C$-Radius}
	\end{subfigure}
	\caption{Radii associated with $\varphi_e(z)=e^z$ and $\varphi_C(z)=1+{4z}/{3}+{2z^2}/{3}$ for the class $\Q$.}
	\label{p1f1}
	\end{figure}
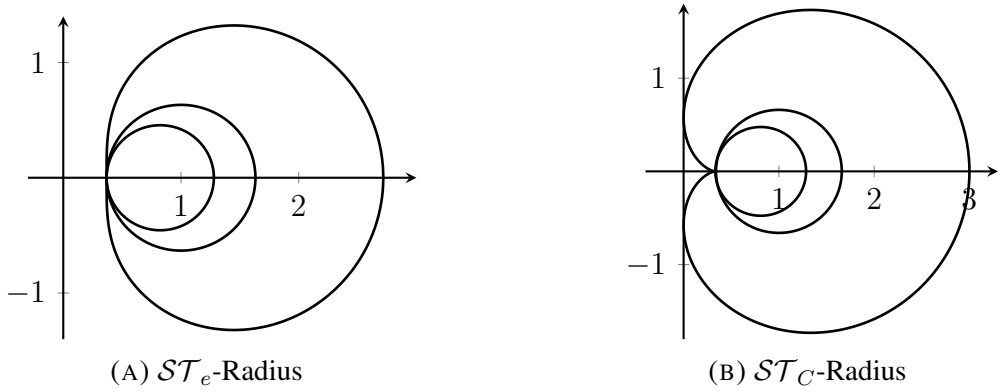
	
	The disk estimates for the starlikeness and convexity problems have the same center and are determined by the same boundary point of the corresponding Ma--Minda domain. Consequently, the associated geometric illustrations are identical for both the starlikeness and convexity radii.
	
	The  class $\mathcal{ST}_{\leftmoon}$  of starlike associated with a lune were proposed by Raina and Sokół~\cite{MR3419845} in 2015 where
		$
		\mathcal{ST}_{\leftmoon}
		=
		\mathcal{ST}\!\left(\varphi_{\leftmoon}\right)
		$
		and
		$
		\mathcal{CV}_{\leftmoon}
		=
		\mathcal{CV}\!\left(\varphi_{\leftmoon}\right),
		$
		with
		$
		\varphi_{\leftmoon}(z)=z+\sqrt{1+z^{2}}.
		$
		The image of the unit disk under $\varphi_{\leftmoon}$ is the lune domain
		$
		\left\{w\in\mathbb{C}:\left|w^{2}-1\right|<2|w|\right\}.
		$

		\begin{lemma}\label{lemma3}\cite[lemma 2.1]{MR3718233}
			For $\sqrt{2}-1<a\leq \sqrt{2}+1,\ \{w:|w-a|<r_a\}\subseteq \{w:|w^2-1|<2|w|\}=\varphi_{\leftmoon}(\mathbb{D})$, where $r_a=1-|\sqrt{2}-a|$.
		\end{lemma}
	\begin{corollary}\label{corollary3}Let
	\[
R_1
	=
	\min\left\{
	1,
	\left(
	\frac{(2-\sqrt2)n(n+1)}
	{M(n+2-\sqrt2)}
	\right)^{1/n}
	\right\},\] and \[
R_2
	=
	\min\left\{
	1,
	\left(
	\frac{(2-\sqrt2)n}
	{M(n+2-\sqrt2)}
	\right)^{1/n}
	\right\}.
	\]	The $\mathcal {ST}_{\leftmoon}$-radius of the class $\Q$ is $R_1$
	and the $\mathcal {CV}_{\leftmoon}$-radius is $R_2$.
	\end{corollary}
	
	\begin{proof}	
	Since $
	\mathbb D\!\left(1;2-\sqrt2\right)\subset\varphi_{\leftmoon}(\mathbb{D})$
	by Lemma~\ref{lemma3}, an application of
	Theorem~\ref{main theorem} with
	\(
	\delta=\sqrt2-1
	\)
	gives the stated lower bounds for the
	$\mathcal {ST}_{\leftmoon}$- and
	$\mathcal{CV}_{\leftmoon}$-radii.
 It remains to establish the sharpness of these bounds. Choose $z_0 \in \D$ with $|z_0|=\rho$ such that $z_0^n=-\rho^n$, consider the extremal function $f_0:\D \rightarrow \mathbb{C}$ defined  by \eqref{eq:sharp-stt}.
 We first consider the
	$\mathcal {ST}_{\leftmoon}$-radius. Let
	\[
	\rho
	=
	\left(
	\frac{(2-\sqrt2)n(n+1)}
	{M(n+2-\sqrt2)}
	\right)^{1/n}
	<1,
	\]
	Using the definition of $\rho$, we have
	$
	M\rho^n(n+2-\sqrt2)
	=
	(2-\sqrt2)n(n+1),
	$
	and consequently,
	\[
	\frac{nM\rho^n}
	{n(n+1)-M\rho^n}
	=1-(\sqrt2-1).
	\]
	Hence, by \eqref{eq:sharp-st}, evaluated at $z=z_0$, we have
	\[
	\frac{z_0f_0'(z_0)}{f_0(z_0)}
	=1-\frac{nM\rho^n}{n(n+1)-M\rho^n}=
	\sqrt2-1
	=
	\varphi_{\leftmoon}(-1)\in
		\partial\varphi_{\leftmoon}(\mathbb{D}).
	\]
	showing that the
	$\mathcal {ST}_{\leftmoon}$-radius is sharp.
	
	The proof for the $\mathcal {CV}_{\leftmoon}$-radius follows in the same manner. Let
	\[
	\rho
	=
	\left(
	\frac{(2-\sqrt2)n}
	{M(n+2-\sqrt2)}
	\right)^{1/n}
	<1.
	\]
	Using the definition of $\rho$, we obtain
	$
	M\rho^n(n+2-\sqrt2)
	=
	(2-\sqrt2)n,
	$
	and hence
	\[
	\frac{nM\rho^n}
	{n-M\rho^n}
	=
	1-(\sqrt2-1).
	\]
	Using \eqref{eq:sharp-cv},  evaluated at $z=z_0$, we have
	\[
	1+\frac{z_0f_0''(z_0)}{f_0'(z_0)}
	=
	1-
	\frac{nM\rho^n}
	{n-M\rho^n}
	=
	\sqrt2-1
	=
	\varphi_{\leftmoon}(-1)\in
		\partial\varphi_{\leftmoon}(\mathbb{D}).
	\]
	showing that the
	$\mathcal {CV}_{\leftmoon}$-radius is sharp.
	\end{proof}

	\begin{remark}
	For the particular class $\mathcal{Q}_1(1)$, corresponding to $n=1$ and $M=1$, the
	radius formulas in Corollary~\ref{corollary3} reduce to
	\[
	R_{\mathcal {ST}_{\leftmoon}}(\mathcal{Q}_1(1))
	=
	\frac{2(2-\sqrt2)}
	{3-\sqrt2}
	\approx
	0.738796,
	\] and \[
	R_{\mathcal {CV}_{\leftmoon}}(\mathcal{Q}_1(1))
	=
	\frac{2-\sqrt2}
	{3-\sqrt2}
	\approx
	0.369398.
	\]
	\end{remark}

	The Ma--Minda classes associated with the rational function were introduced by Kumar and Ravichandran~\cite{MR3496681} and correspond to the function
	$
	\varphi_{R}(z)
	=
	1+({z^{2}+kz})/({k^{2}-kz}),
	$ where $
	k=\sqrt{2}+1.
	$
	Specifically,
	$
	\mathcal{ST}_{R}
	=
	\mathcal{ST}(\varphi_{R})
	$ and $
	\mathcal{CV}_{R}
	=
	\mathcal{CV}(\varphi_{R}).
	$

	\begin{lemma}\label{lemma4}\cite[Lemma 2.2]{MR3496681}
			
			For $2(\sqrt{2}-1)<a<2,\ \{w:|w-a|<r_a\}\subseteq \varphi_{R}(\mathbb{D})$, where $r_a$ is
			\begin{align*}
				r_a &=
				\begin{dcases}
					a-2(\sqrt{2}-1),\ & \ 2(\sqrt{2}-1)<a\leq \sqrt{2}\\
					2-a,\ & \sqrt{2}\leq a<2.
				\end{dcases}
			\end{align*}
		\end{lemma}
\begin{corollary}\label{corollary4}
Let
\[
R_1
=
\min\left\{
1,
\left(
\frac{(3-2\sqrt2)n(n+1)}
{M(n+3-2\sqrt2)}
\right)^{1/n}
\right\},\] and \[
R_2
=
\min\left\{
1,
\left(
\frac{(3-2\sqrt2)n}
{M(n+3-2\sqrt2)}
\right)^{1/n}
\right\}.
\] The $\mathcal {ST}_{R}$-radius of the class $\Q$ is $R_1$
and the $\mathcal {CV}_{R}$-radius is $R_2$.
\end{corollary}

\begin{proof}
From Lemma~\ref{lemma4},
$
\mathbb D\!\left(1;3-2\sqrt2\right)\subset\varphi_R(\mathbb{D}).
$
Applying Theorem~\ref{main theorem} with
\(
\delta=2(\sqrt2-1)
\)
gives the asserted lower bounds for the
$\mathcal {ST}_R$- and $\mathcal {CV}_R$-radii.
\noindent It remains to establish the sharpness of these bounds. Choose $z_0 \in \D$ with $|z_0|=\rho$ such that $z_0^n=-\rho^n$, consider the extremal function $f_0:\D \rightarrow \mathbb{C}$ defined  by \eqref{eq:sharp-stt}.
We next verify that these bounds are sharp. Consider first the
$\mathcal {ST}_R$-radius. Let
\[
\rho
=
\left(
\frac{(3-2\sqrt2)n(n+1)}
{M(n+3-2\sqrt2)}
\right)^{1/n}
<1,
\]
By the definition of $\rho$ it follows $ M\rho^n(n+3-2\sqrt2)
=(3-2\sqrt2)n(n+1),$
which implies
\[
\frac{nM\rho^n}
{n(n+1)-M\rho^n}
=
3-2\sqrt2.
\]
Using \eqref{eq:sharp-st} and evaluated at $z=z_0$, hence
\[
	\frac{z_0f_0'(z_0)}{f_0(z_0)}
	=1-\frac{nM\rho^n}{n(n+1)-M\rho^n}
=
2(\sqrt2-1)
=
\varphi_R(-1)\in
\partial\varphi_R(\mathbb{D}).
\]
which shows that the
$\mathcal {ST}_R$-radius is sharp.

An analogous argument establishes the sharpness of the
$\mathcal {CV}_R$-radius. Let
\[
\rho
=
\left(
\frac{(3-2\sqrt2)n}
{M(n+3-2\sqrt2)}
\right)^{1/n}
<1.
\]
Again, by the definition of $\rho$, we have $ M\rho^n(n+3-2\sqrt2)
=(3-2\sqrt2)n,$
and therefore
\[
\frac{nM\rho^n}
{n-M\rho^n}
=
3-2\sqrt2.
\]
Using \eqref{eq:sharp-cv} and evaluated at $z=z_0$, hence
$
|z_0|=\rho,
$ and $
z_0^n=-\rho^n,
$
we obtain
\[
1+\frac{z_0f_0''(z_0)}{f_0'(z_0)}
=
1-
\frac{nM\rho^n}
{n-M\rho^n}
=
2(\sqrt2-1)
=
\varphi_R(-1)\in
\partial\varphi_R(\mathbb{D}).
\]
which proves that the
$\mathcal {CV}_R$-radius is sharp.
\end{proof}

	\begin{remark}
	For the particular class $\mathcal{Q}_1(1)$, corresponding to $n=1$ and $M=1$, the
	radius formulas in Corollary~\ref{corollary4} reduce to
	\[
	R_{\mathcal {ST}_R}(\mathcal Q_1(1))
	=
	\frac{2(3-2\sqrt2)}
	{4-2\sqrt2}
	=
	\frac{2-\sqrt2}{2}
	\approx
	0.292893,
	\]
	and
	\[
	R_{\mathcal {CV}_R}(\mathcal Q_1(1))
	=
	\frac{3-2\sqrt2}
	{4-2\sqrt2}
	=
	\frac{2-\sqrt2}{4}
	\approx
	0.146447.
	\]
	\end{remark}
	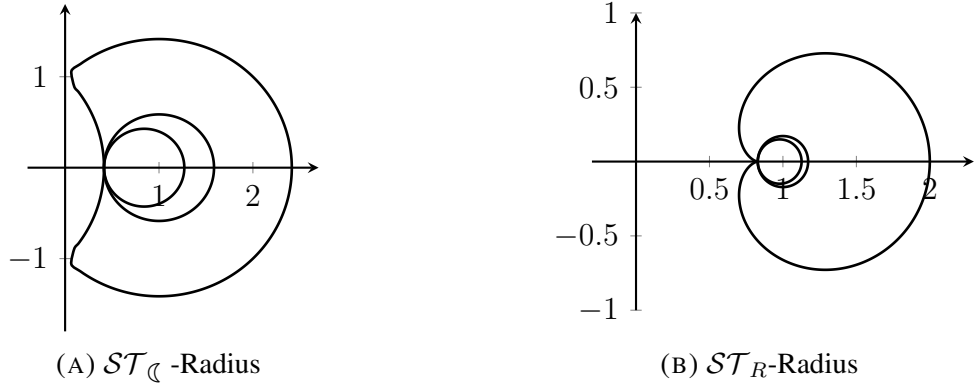
\begin{figure}[!htbp]
		\centering
		\begin{subfigure}{0.48\textwidth}
		\centering
	\begin{tikzpicture}
			\begin{axis}[
			    axis lines=center,
			    width=8cm,
			    height=8.8cm,
			    xmin=-0.4, xmax=2.7,
			    ymin=-1.8, ymax=1.8,
			    xtick={1.0,2.0,3.0},
			    ytick={-1,1},
			    thick,
			    scale=0.6
			]
			   \addplot[
			   			domain=0:360,
			   			samples=400,
			   			smooth,
			   			color=black,
			   			line width=1pt
			   		]
			   		({0.8419 + 0.4277*cos(x)},
			   		{0.4277*sin(x)});
			
			\addplot[
						domain=0:360,
						samples=400,
						smooth,
					color=black,
						line width=1pt
					]
					({1 + 0.585786*cos(x)},
					{0.585786*sin(x)});
			\addplot[
			    domain=-180:180,
			    samples=400,
			    smooth,
			    color=black,
			    line width=1pt,
			]
			({cos(x) + sqrt(abs(cos(x)) + (cos(x))^2)},
			 {sin(x) + sign(sin(2*x)) * sqrt(abs(cos(x)) - (cos(x))^2)});	
			\end{axis}
			\end{tikzpicture}
	
		\caption{$\mathcal{ST}_{\leftmoon}$-Radius}
		\end{subfigure}
		\begin{subfigure}{0.48\textwidth}
		\centering
		
	\begin{tikzpicture}
			\begin{axis}[
			    axis lines=center,
			    width=8.8cm,
			    height=7.2cm,
			    xmin=-0.3, xmax=2.3,
			    ymin=-1, ymax=1,
			    xtick={0.5,1.0,1.5,2.0},
			    ytick={-1.0,-0.5,0.5,1.0},
			    thick,
			    scale=0.7
			]
		   \addplot[
		   			domain=0:360,
		   			samples=400,
		   			smooth,
		   			color=black,
		   			line width=1pt
		   		]
		   		({0.978 + 0.149*cos(x)},
		   		{0.149*sin(x)});
		
		\addplot[
					domain=0:360,
					samples=400,
					smooth,
				color=black,
					line width=1pt
				]
				({1 + 0.171573*cos(x)},
				{0.171573*sin(x)});		
			\addplot[
			    domain=-180:180,
			    samples=400,
			    smooth,
			    color=black,
			    line width=1pt,
			]
			({(cos(2*x) - 2*sqrt(2)*cos(x) + 3 + 2*sqrt(2)) / den(x)},
			 {(sin(2*x) + 2*sin(x)) / den(x)});
			
			\end{axis}
			\end{tikzpicture}
		
		\caption{$\mathcal {ST}_R$-Radius}
		\end{subfigure}
		\caption{Radii associated with $\varphi_{\leftmoon}(z)=z+\sqrt{1+z^2}$ and $\varphi_R(z)=1+(z^2+kz)/(k^2-kz) $, $k=\sqrt{2}+1$ for the class $\Q$.}
		\label{p1f2}
			\end{figure}
	
		The cardioid starlike class were introduced by Kumar and Kamaljeet~\cite{MR4216363} as the Ma--Minda class corresponding to the function
	$
	\varphi_{c}(z)=1+ze^{z}.
	$
	Specifically,
	$
	\mathcal{ST}_{c}
	=
	\mathcal{ST}(\varphi_{c})
	$ and $\mathcal{CV}_{c}
	=
	\mathcal{CV}(\varphi_{c})$.
	\begin{lemma}\label{lemma5}\cite[Lemma 2.2]{MR4216363}
			
		For $1-1/e<a<1+e$, $\{w:|w-a|<r_a\}\subseteq \varphi_{c}(\mathbb{D})$, where $r_a$ is
		\begin{align*}
			r_a &=
			\begin{dcases}
			a-1+\frac{1}{e},\ & 1-\frac{1}{e}<a\le 1+\frac{e-e^{-1}}{2} \ \\
			e-(a-1),\ & 1+\frac{e-e^{-1}}{2}\le a<1+e.
			\end{dcases}
		\end{align*}
	\end{lemma}	
\begin{corollary}\label{corollary5}
	Let
	\[
R_1
	=
	\min\left\{
	1,
	\left(
	\frac{n(n+1)}
	{M(1+en)}
	\right)^{1/n}
	\right\},
\quad \text{and}\quad
R_2
	=
	\min\left\{
	1,
	\left(
	\frac{n}
	{M(1+en)}
	\right)^{1/n}
	\right\}.
	\] The $\mathcal {ST}_{c}$-radius of the class $\Q$ is $R_1$
	and the $\mathcal {CV}_{c}$-radius is $R_2$.
\end{corollary}

\begin{proof}
	By Lemma~\ref{lemma5},
	$
	\mathbb D\!\left(1;1/e\right)\subset\varphi_c(\mathbb{D}).
	$
	Therefore, applying Theorem~\ref{main theorem} with $\delta=1-1/e$, we obtain the asserted lower
	bounds for the $\mathcal {ST}_c$- and $\mathcal {CV}_c$-radii.
	It remains to prove sharpness. Choose $z_0 \in \D$ with $|z_0|=\rho$ such that $z_0^n=-\rho^n$, consider the extremal function $f_0:\D \rightarrow \mathbb{C}$ defined by \eqref{eq:sharp-stt}.
	First, consider the
	$\mathcal {ST}_c$-radius. Let
	\[
	\rho
	=
	\left(
	\frac{n(n+1)}
	{M(1+en)}
	\right)^{1/n}
	<1,
	\]
	Since $	M\rho^n(1+en)
	=
	n(n+1), $	we obtain
	\[
	\frac{nM\rho^n}
	{n(n+1)-M\rho^n}
	=
	\frac1e.
	\]
	Using \eqref{eq:sharp-st} and evaluated at $z=z_0$, hence
	\[
	\frac{z_0f_0'(z_0)}{f_0(z_0)}
	=
	1-
	\frac{nM\rho^n}
	{n(n+1)-M\rho^n}
	=
	1-\frac1e
	=
	\varphi_c(-1)	\in
		\partial\varphi_c(\mathbb{D}).
	\]
	which proves that the
	$\mathcal {ST}_c$-radius is sharp.
	
	The sharpness of the $\mathcal {CV}_c$-radius is established similarly. Let
	\[
	\rho
	=
	\left(
	\frac{n}
	{M(1+en)}
	\right)^{1/n}
	<1.
	\]
	Since $
	M\rho^n(1+en)
	=
	n,$
	it follows that
	\[
	\frac{nM\rho^n}
	{n-M\rho^n}
	=
	\frac1e.
	\]
	Using \eqref{eq:sharp-cv} and evaluated at $z=z_0$, hence
	\[
	1+\frac{z_0f_0''(z_0)}{f_0'(z_0)}
	=
	1-
	\frac{nM\rho^n}
	{n-M\rho^n}
	=
	1-\frac1e
	=
	\varphi_c(-1)	\in
			\partial\varphi_c(\mathbb{D}).
	\]
	which proves that  the
	$\mathcal {CV}_c$-radius is also sharp.
\end{proof}

\begin{remark}
	For the particular class $\mathcal{Q}_1(1)$, corresponding to $n=1$ and $M=1$, the
	radius formulas in Corollary~\ref{corollary5} reduce to
	\[
	R_{\mathcal {ST}_c}(\mathcal Q_1(1))
	=
	\frac{2}{1+e}
	\approx
	0.537883,
	\quad \text{and}\quad
	R_{\mathcal {CV}_c}(\mathcal Q_1(1))
	=
	\frac{1}{1+e}
	\approx
	0.268941.
	\]
\end{remark}
	
		In 2021, Gupta \textit{et al.}~\cite{MR4303808} introduced the cardioid Ma--Minda classes corresponding to the function
	$
	\varphi_{car}(z)=1+z+{z^{2}}/{2}.
	$
	Specifically, $\mathcal{ST}_{car}=\mathcal{ST}(\varphi_{car})$ and $\mathcal{CV}_{car}=\mathcal{CV}(\varphi_{car})$. The image of the unit disk under $\varphi_{car}$ is the domain bounded by the cardioid
	$
	\left(4x^{2}+4y^{2}-8x-1\right)^{2}
	+
	4\left(4x^{2}+4y^{2}-12x+1\right)
	=0.
	$
	
		\begin{lemma}\label{lemma6}\cite[Lemma 2.2]{MR4303808}
			For $1/2<a< 5/2,\ \{w:|w-a|<r_a\}\subseteq \varphi_{car}(\mathbb{D})$, where $r_a$ is
			\begin{align*}
			r_a &=
			\begin{dcases}
			a-\frac{1}{2},\ & \ \frac{1}{2}<a\leq \frac{3}{2}\\
			\frac{5}{2}-a,\ & \frac{3}{2}\leq a<\frac{5}{2}.
			\end{dcases}
			\end{align*}
		\end{lemma}
		
	\begin{corollary}\label{corollary6}
	Let
	\[
	R_1
	=
	\min\left\{
	1,
	\left(
	\frac{n(n+1)}
	{M(2n+1)}
	\right)^{1/n}
	\right\},
\quad \text{and }\quad
R_2
	=
	\min\left\{
	1,
	\left(
	\frac{n}
	{M(2n+1)}
	\right)^{1/n}
	\right\}.
	\]The $\mathcal {ST}_{car}$-radius of the class $\Q$ is  $R_1$
	and the $\mathcal {CV}_{car}$-radius $R_2$.
	\end{corollary}
	
	\begin{proof}
	By Lemma~\ref{lemma6},
	$
	\mathbb D\!\left(1;1/2\right)\subset\varphi_{car}(\mathbb{D}).
	$
	Hence, Theorem~\ref{main theorem} with
	\(
	\delta=1/2
	\)
	immediately yields the asserted lower bounds for the
	$\mathcal {ST}_{car}$- and $\mathcal {CV}_{car}$-radii.
		It remains to establish the sharpness of these bounds.
	Choose $z_0 \in \D$ with $|z_0|=\rho$ such that $z_0^n=-\rho^n$, consider the extremal function $f_0:\D \rightarrow \mathbb{C}$ defined by \eqref{eq:sharp-stt}.
			 We first consider the
	$\mathcal {ST}_{car}$-radius. Let
	\[
	\rho
	=
	\left(
	\frac{n(n+1)}
	{M(2n+1)}
	\right)^{1/n}
	<1,
	\]
	Using the definition of $\rho$, we have $M\rho^n(2n+1)
	=
	n(n+1),$
	and consequently,
	\[
	\frac{nM\rho^n}
	{n(n+1)-M\rho^n}
	=
	\frac12.
	\]
	Evaluating \eqref{eq:sharp-st} at $z=z_0$ gives
	\[
	\frac{z_0f_0'(z_0)}{f_0(z_0)}
	=
	1-
	\frac{nM\rho^n}
	{n(n+1)-M\rho^n}
	=
	\frac12
	=
	\varphi_{car}(-1)	\in
		\partial\varphi_{car}(\mathbb{D}).
	\]
	this showing that  the
	$\mathcal {ST}_{car}$-radius is  sharp.
	
	The proof for the $\mathcal {CV}_{car}$-radius follows in the same manner. Let
	\[
	\rho
	=
	\left(
	\frac{n}
	{M(2n+1)}
	\right)^{1/n}
	<1.
	\]
	Using the definition of $\rho$, we obtain $
	M\rho^n(2n+1)
	=
	n,$
	and hence
	\[
	\frac{nM\rho^n}
	{n-M\rho^n}
	=
	\frac12.
	\]
	Evaluating \eqref{eq:sharp-cv} at $z=z_0$,
	\[
	1+\frac{z_0f_0''(z_0)}{f_0'(z_0)}
	=
	1-
	\frac{nM\rho^n}
	{n-M\rho^n}
	=
	\frac12
	=
	\varphi_{car}(-1)\in
		\partial\varphi_{car}(\mathbb{D}).
	\]
	showing that
	$\mathcal {CV}_{car}$-radius is sharp.
	\end{proof}
	\begin{remark}
	For the particular class $\mathcal{Q}_1(1)$, corresponding to $n=1$ and $M=1$, the
	radius formulas in Corollary~\ref{corollary6} reduce to
	\[
	R_{\mathcal {ST}_{car}}(\mathcal{Q}_1(1))
	=
	\frac{2}{3}
	=
	0.666667,
\quad \text{	and }	\quad
	R_{\mathcal {CV}_{car}}(\mathcal{Q}_1(1))
	=
	\frac{1}{3}
	\approx
	0.333333.
	\]
	\end{remark}
	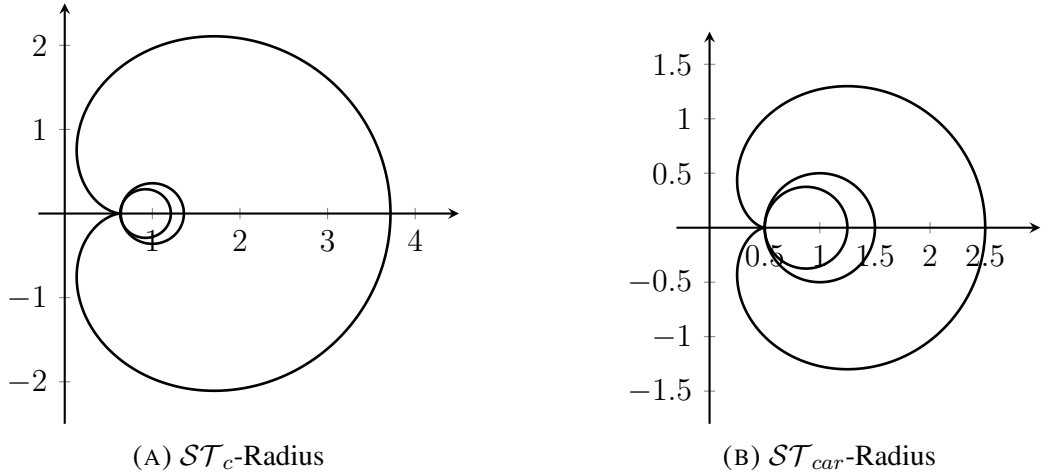
\begin{figure}[!htbp]
		\centering
		\begin{subfigure}{0.48\textwidth}
		\centering
		\begin{tikzpicture}
			\begin{axis} [
			    axis lines=center,
			    width=9cm,
			    height=9cm,
			    xmin=-0.3, xmax=4.5,
			    ymin=-2.5, ymax=2.5,
			    xtick={-2.0,-1.0,1.0,2.0,3.0,4.0},
			    ytick={-3.0,-2.0,-1.0,1.0,2.0,3.0},
			    thick,
			    scale=0.75
			]
		   \addplot[
		   			domain=0:360,
		   			samples=400,
		   			smooth,
		   			color=black,
		   			line width=1pt
		   		]
		   		({0.922+ 0.289*cos(x)},
		   		{0.289*sin(x)});
		
		\addplot[
					domain=0:360,
					samples=400,
					smooth,
				color=black,
					line width=1pt
				]
				({1 + 0.36*cos(x)},
				{0.36*sin(x)});		
			\addplot[
			    domain=-180:180,
			    samples=400,
			    smooth,
			    color=black,
			    line width=1pt,
			]
			({1 + expx(x) * cos(x + deg(sin(x)))},
			 {expx(x) * sin(x + deg(sin(x)))});
			
			\end{axis}
		\end{tikzpicture}
		\caption{$\mathcal {ST}_c$-Radius}
		\end{subfigure}
		\begin{subfigure}{0.48\textwidth}
		\centering
	\begin{tikzpicture}
		\begin{axis} [
		    axis lines=center,
		    width=8cm,
		    height=8.5cm,
		    xmin=-0.3, xmax=3.0,
		    ymin=-1.8, ymax=1.8,
		    xtick={0.5,1.0,1.5,2.0,2.5},
		    ytick={-1.5,-1.0,-0.5,0.5,1.0,1.5},
		    thick,
		    scale=0.75
		]
		   \addplot[
		   			domain=0:360,
		   			samples=400,
		   			smooth,
		   			color=black,
		   			line width=1pt
		   		]
		   		({0.875 + 0.375*cos(x)},
		   		{0.375*sin(x)});
		
		\addplot[
					domain=0:360,
					samples=400,
					smooth,
				color=black,
					line width=1pt
				]
				({1 + 0.5*cos(x)},
				{0.5*sin(x)});	
		\addplot[
		    domain=-180:180,
		    samples=400,
		    smooth,
		    color=black,
		    line width=1pt,
		]
		({1 + cos(x) + 0.5 * cos(2*x)},
		 {sin(x) + 0.5 * sin(2*x)});
		
		\end{axis}
		\end{tikzpicture}
		\caption{$\mathcal{ST}_{car}$-Radius}
		\end{subfigure}
		\caption{Radii associated with $\varphi_c(z)=1+ze^{z}$ and  $\varphi_{car}(z)=1+z+{z^2}/{2}$ for the class $\Q$.}
		\label{p1f3}
			\end{figure}
	
	Masih \textit{et al.}~\cite{Masih} studied   Ma--Minda classes corresponding to the function
	$
	\varphi_{s}(z)=(1+sz)^2,\, 0<s\leq 1/\sqrt{2}.
	$
	In the present work, we consider the extremal case
	$
	\varphi_{lim}(z)=\left(1+{z}/{\sqrt{2}}\right)^2
	$, which corresponds to $s=1/\sqrt{2}$. Specifically, $\mathcal{ST}_{lim}=\mathcal{ST}(\varphi_{lim})$ and $\mathcal{CV}_{lim}=\mathcal{CV}(\varphi_{lim})$. We shall use the inclusion lemma established by Kanaga and Ravichandran~\cite{MR4478088} to derive the corresponding radius results(see also \cite{MR4386722}).

	\begin{lemma}\label{lemma7}\cite{MR4478088}
			For $(1-1/\sqrt{2})^2<a< (1+1\sqrt{2})^2,\ \{w:|w-a|<r_a\}\subseteq \varphi_{lim}(\mathbb{D})$, where $r_a$ is
			\begin{align*}
				r_a &=
				\begin{dcases}
					a-\left(1-\frac{1}{\sqrt{2}}\right)^2,\ & \ \left(1-\frac{1}{\sqrt{2}}\right)^2<a\leq \frac{3}{2}\\
					\left(1+\frac{1}{\sqrt{2}}\right)^2-a,\ & \frac{3}{2}\leq a<\left(1+\frac{1}{\sqrt{2}}\right)^2.
				\end{dcases}
			\end{align*}
		\end{lemma}		
	\begin{corollary}\label{corollary7}
Let
	\[
	R_1
	=
	\min\left\{
	1,
	\left(
	\frac{(2\sqrt2-1)n(n+1)}
	{M(2n-1+2\sqrt2)}
	\right)^{1/n}
	\right\},\] and \[
R_2
	=
	\min\left\{
	1,
	\left(
	\frac{(2\sqrt2-1)n}
	{M(2n-1+2\sqrt2)}
	\right)^{1/n}
	\right\}.
	\]	The $\mathcal {ST}_{lim}$-radius of the class $\Q$ is $R_1$
	and the $\mathcal {CV}_{lim}$-radius $R_2$.	\end{corollary}
	
	\begin{proof}
	From Lemma~\ref{lemma7},
	$
	\mathbb D\!\left(1;\sqrt2-1/2\right)\subset\varphi_{lim}(\mathbb{D})
	$.
	Applying Theorem~\ref{main theorem} with
	\(
	\delta=\left(1-1/{\sqrt2}\right)^2
	\)
	gives the asserted lower bounds for the
	$\mathcal {ST}_{lim}$- and $\mathcal {CV}_{lim}$-radii.
	We next verify that these bounds are sharp. 	Choose $z_0 \in \D$ with $|z_0|=\rho$ such that $z_0^n=-\rho^n$, consider the extremal function $f_0:\D \rightarrow \mathbb{C}$ defined by
\eqref{eq:sharp-stt}. Consider the
	$\mathcal {ST}_{lim}$-radius. Let
	\[
	\rho
	=
	\left(
	\frac{(2\sqrt2-1)n(n+1)}
	{M(2n-1+2\sqrt2)}
	\right)^{1/n}
	<1,
	\]
	From the definition of $\rho$, it follows $
	M\rho^n(2n-1+2\sqrt2)
	=
	(2\sqrt2-1)n(n+1),$
	which implies
	\[
	\frac{nM\rho^n}
	{n(n+1)-M\rho^n}
	=
	\sqrt2-\frac12.
	\]
	By substituting $z=z_0$ into \eqref{eq:sharp-st}, we obtain
	\[
	\frac{z_0f_0'(z_0)}{f_0(z_0)}
	=
	1-
	\frac{nM\rho^n}
	{n(n+1)-M\rho^n}
	=
	\frac32-\sqrt2
	=
	\varphi_{lim}(-1)\in
		\partial\varphi_{lim}(\mathbb{D}).
	\]
	which shows that the
	$\mathcal {ST}_{lim}$-radius is  sharp.
	
	An analogous argument establishes the sharpness of the
	$\mathcal {CV}_{lim}$-radius. Let
	\[
	\rho
	=
	\left(
	\frac{(2\sqrt2-1)n}
	{M(2n-1+2\sqrt2)}
	\right)^{1/n}
	<1.
	\]
	Again, by the definition of $\rho$, we have $
	M\rho^n(2n-1+2\sqrt2)
	=
	(2\sqrt2-1)n,$
	and therefore
	\[
	\frac{nM\rho^n}
	{n-M\rho^n}
	=
	\sqrt2-\frac12.
	\]
	Evaluating \eqref{eq:sharp-cv} at $z=z_0$ gives
	\[
	1+\frac{z_0f_0''(z_0)}{f_0'(z_0)}
	=
	1-
	\frac{nM\rho^n}
	{n-M\rho^n}
	=
	\frac32-\sqrt2
	=
	\varphi_{lim}(-1)\in
		\partial\varphi_{lim}(\mathbb{D}).
	\]
which proves that the
	$\mathcal {CV}_{lim}$-radius   is sharp.
	\end{proof}

	\begin{remark}
	For the particular class $\mathcal{Q}_1(1)$, corresponding to $n=1$ and $M=1$, the
	radius formulas in Corollary~\ref{corollary7} reduce to
	\[
	R_{\mathcal {ST}_{lim}}(\mathcal{Q}_1(1))
	=
	\frac{4\sqrt2-2}{2\sqrt2+1}
	\approx
	0.955185,
\] and \[
	R_{\mathcal {CV}_{lim}}(\mathcal{Q}_1(1))
	=
	\frac{2\sqrt2-1}{2\sqrt2+1}
	\approx
	0.477592.
	\]
	\end{remark}

	Cho \textit{et al.}~\cite{MR3913990} introduced the Ma--Minda starlike class associated with the sine function
	$
	\varphi_{\sin}(z)=1+\sin z.
	$
	Specifically, $\mathcal{ST}_{\sin}=\mathcal{ST}(\varphi_{\sin})$ and $\mathcal{CV}_{\sin}=\mathcal{CV}(\varphi_{\sin})$.

		\begin{lemma}\label{lemma8}\cite[Lemma 3.3]{MR3913990}
			For $1-\sin 1\le a \le 1+\sin1$, If $ r_{a}=\sin 1-\left| a-1\right|$ then
			\[ \left\{w:\left|w-a\right|<r_{a}\right\} \subseteq \varphi_{\sin}(\mathbb{D})\]
		\end{lemma}
	\begin{corollary}\label{corollary8}
Let
	\[
R_1
	=
	\min\left\{
	1,
	\left(
	\frac{(\sin 1)n(n+1)}
	{M(n+\sin 1)}
	\right)^{1/n}
	\right\},
\] and \[
	R_2
	=
	\min\left\{
	1,
	\left(
	\frac{(\sin 1)n}
	{M(n+\sin 1)}
	\right)^{1/n}
	\right\}.
	\]The $ \mathcal {ST}_{\sin}$-radius of the class $\Q$ is $R_1$
	and the $\mathcal {CV}_{\sin}$-radius is $R_2$.
	\end{corollary}
	
	\begin{proof}
	By Lemma~\ref{lemma8},
	$
	\mathbb D\!\left(1;\sin1\right)\subset\varphi_{\sin}(\mathbb{D})
	$.
	Hence, applying Theorem~\ref{main theorem} with
	\(
	\delta=1-\sin1
	\)
	yields the asserted lower bounds for the
	$\mathcal {ST}_{\sin}$- and $\mathcal {CV}_{\sin}$-radii.
 It only remains to verify the sharpness of these bounds. Choose $z_0\in \D$ satisfying
			$
			|z_0|=\rho,
			$ where $
			z_0^n=-\rho^n.
			$ Consider the extremal function $f_0:\D \rightarrow \mathbb C$ defined by
	\eqref{eq:sharp-stt}. We first deal with the
	$\mathcal {ST}_{\sin}$-radius. Let
	\[
	\rho
	=
	\left(
	\frac{(\sin1)n(n+1)}
	{M(n+\sin1)}
	\right)^{1/n}
	<1,
	\]
	we obtain $
		M\rho^n(n+\sin1)
		=
		(\sin1)n(n+1),$
		and hence
		\[
		\frac{nM\rho^n}
		{n(n+1)-M\rho^n}
		=
		\sin1.
		\]
	Substitute $z=z_0$ in \eqref{eq:sharp-st}
	\[
	\frac{z_0f_0'(z_0)}{f_0(z_0)}
	=
	1-
		\frac{nM\rho^n}
		{n(n+1)-M\rho^n}
	=
	1-\sin1
	=
	\varphi_{\sin}(-1)\in
		\partial\varphi_{\sin}(\mathbb{D}).
	\]
	which establishes that the
	$\mathcal {ST}_{\sin}$-radius is sharp.
	
	The argument for the
	$\mathcal {CV}_{\sin}$-radius is entirely similar. Let
	\[
	\rho
	=
	\left(
	\frac{(\sin1)n}
	{M(n+\sin1)}
	\right)^{1/n}
	<1.
	\]
	we obtain $
	M\rho^n(n+\sin1)
	=
	(\sin1)n, $
	so that
	\[
	\frac{nM\rho^n}
	{n-M\rho^n}
	=
	\sin1.
	\]
	Substituting $z=z_0$ into \eqref{eq:sharp-cv}, we obtain
	\[
	1+\frac{z_0f_0''(z_0)}{f_0'(z_0)}
	=
	1-
	\frac{nM\rho^n}
	{n-M\rho^n}
	=
	1-\sin1
	=
	\varphi_{\sin}(-1)\in
		\partial\varphi_{\sin}(\mathbb{D}).
	\]
	which establishes the sharpness of  the
	$\mathcal{CV}_{\sin}$-radius.
	\end{proof}
	
	\begin{remark}
	For the particular class $\mathcal{Q}_1(1)$, corresponding to $n=1$ and $M=1$, the
	radius formulas in Corollary~\ref{corollary8} reduce to
	\[
	R_{\mathcal {ST}_{\sin}}(\mathcal{Q}_1(1))
	=
	\frac{2\sin 1}{1+\sin 1}
	\approx
	0.913912,
\] and \[
	R_{\mathcal {CV}_{\sin}}(\mathcal{Q}_1(1))
	=
	\frac{\sin 1}{1+\sin 1}
	\approx
	0.456956.
	\]
	\end{remark}
	\begin{figure}[!htbp]
		\centering
		\begin{subfigure}{0.48\textwidth}
		\centering
		
\begin{tikzpicture}
		\begin{axis} [
		    axis lines=center,
		    width=7.2cm,
		    height=8.5cm,
		    xmin=-0.3, xmax=3.2,
		    ymin=-2.2, ymax=2.2,
		    xtick={1.0,2.0,3.0},
		    ytick={-2.0,-1.0,1.0,2.0},
		    thick,
		    scale=0.75
		]
		   \addplot[
		   			domain=0:360,
		   			samples=400,
		   			smooth,
		   			color=black,
		   			line width=1pt
		   		]
		   		({0.704 + 0.618*cos(x)},   		{0.618*sin(x)});
		
		\addplot[
					domain=0:360,
					samples=400,
					smooth,
				color=black,
					line width=1pt
				]
				({1 + 0.914214*cos(x)},
				{0.914214*sin(x)});	
		\addplot[
		    domain=-180:180,
		    samples=400,
		    smooth,
		    color=black,
		    line width=1pt,
		]
		({1 + sqrt(2) * cos(x) + 0.5 * cos(2*x)},
		 {sqrt(2) * sin(x) + 0.5 * sin(2*x)});
		
		\end{axis}
		\end{tikzpicture}
		\caption{$\mathcal{ST}_{lim}$-Radius}
		\end{subfigure}
		\begin{subfigure}{0.48\textwidth}
		\centering
		
	\begin{tikzpicture}
		\begin{axis} [
		    axis lines=center,
		    width=8.2cm,
		    height=8.5cm,
		    xmin=-0.3, xmax=2.5,
		    ymin=-1.5, ymax=1.5,
		    xtick={0.5,1.0,1.5,2.0},
		    ytick={-1.0,-0.5,0.5,1.0},
		    thick,
		    scale=0.75
		]
		   \addplot[
		   			domain=0:360,
		   			samples=400,
		   			smooth,
		   			color=black,
		   			line width=1pt
		   		]
		   		({0.736 + 0.577*cos(x)},
		   		{0.577*sin(x)});
		
		\addplot[
					domain=0:360,
					samples=400,
					smooth,
				color=black,
					line width=1pt
				]
				({1 + 0.841471*cos(x)},
				{0.841471*sin(x)});	
		\addplot[
		    domain=-180:180,
		    samples=400,
		    smooth,
		    color=black,
		    line width=1pt,
		]
		({1 + sin(deg(cos(x))) * cosh(sin(x))},
		 {cos(deg(cos(x))) * sinh(sin(x))});
		
		\end{axis}
		\end{tikzpicture}
		\caption{$\mathcal{ST}_{\sin}$-Radius}
		\end{subfigure}
		\caption{Radii associated with $\varphi_{lim}(z)=(1+z/\sqrt2)^2$  and $\varphi_{\sin}(z)=1+\sin z$ for the class $\Q$.}
		\label{p1f4}
		\end{figure}
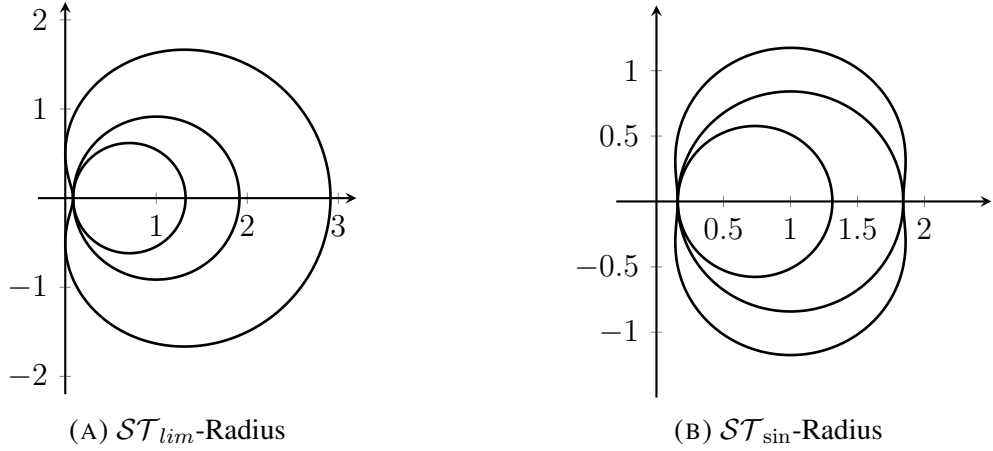
	In 2021, Wani and Swaminathan~\cite{MR4190740} introduced the Ma--Minda classes associated with the nephroid domain
	$
	\varphi_{Ne}(z)=1+z-{z^{3}}/{3}.
	$
	Specifically, $\mathcal{ST}_{Ne}=\mathcal{ST}(\varphi_{Ne})$ and $\mathcal{CV}_{Ne}=\mathcal{CV}(\varphi_{Ne})$. The image of the unit disk under $\varphi_{Ne}$ is the interior of the nephroid bounded by the two-cusped curve
	\[
	\left((u-1)^2+v^2-\frac{4}{9}\right)^3-\frac{4v^2}{3}=0.
	\]
	
		\begin{lemma}\label{lemma9}\cite[Lemma 2.2]{MR4190740}
			For $1/3<a<5/3,\ \{w:|w-a|<r_a\}\subseteq \varphi_{Ne}(\mathbb{D})$, where $r_a$ is
			\begin{align*}
				r_a &=
				\begin{dcases}
					a-\frac{1}{3},\ & \ \frac{1}{3}<a\leq 1,\\
					\frac{5}{3}-a,\ & \ 1\leq a<\frac{5}{3}.
				\end{dcases}
			\end{align*}
		\end{lemma}	
	\begin{corollary}\label{corollary9}Let
	\[
R_1
	=
	\min\left\{
	1,
	\left(
	\frac{2n(n+1)}
	{M(3n+2)}
	\right)^{1/n}
	\right\},
	\quad \text{and}\quad  R_2
	=
	\min\left\{
	1,
	\left(
	\frac{2n}
	{M(3n+2)}
	\right)^{1/n}
	\right\}.
	\]
	The $\mathcal {ST}_{Ne}$-radius of the class $\Q$ is $R_1$
	and the $\mathcal {CV}_{Ne}$-radius is $R_2$.	\end{corollary}
	
	\begin{proof}
	Since
	$
	\mathbb D\!\left(1;2/3\right)\subset\varphi_{Ne}(\mathbb{D})
	$
	by Lemma~\ref{lemma9}, Theorem~\ref{main theorem} with
	\(
	\delta=1/3
	\)
	gives the asserted lower bounds for the
	$\mathcal {ST}_{Ne}$- and $\mathcal {CV}_{Ne}$-radii.
 	We now establish the sharpness of these bounds. Choose $z_0 \in \D$ satisfying
		$
		|z_0|=\rho
		 $ with $
		z_0^n=-\rho^n,
		$ and  take the extremal function $f_0:\D \rightarrow \mathbb C$ defined by
	\eqref{eq:sharp-stt}. Consider the
	$\mathcal {ST}_{Ne}$-radius. Let
	\[
	\rho
	=
	\left(
	\frac{2n(n+1)}
	{M(3n+2)}
	\right)^{1/n}
	<1,
	\]
	Using the defining equation for $\rho$, we have $
	M\rho^n(3n+2)
	=
	2n(n+1), $
	which yields
	\[
	\frac{nM\rho^n}
	{n(n+1)-M\rho^n}
	=
	\frac23.
	\]
	Substituting $z=z_0$ into \eqref{eq:sharp-st}, we obtain
	\[
	\frac{z_0f_0'(z_0)}{f_0(z_0)}
	=
	1-
	\frac{nM\rho^n}
	{n(n+1)-M\rho^n}
	=
	\frac13
	=
	\varphi_{Ne}(-1)\in
		\partial\varphi_{Ne}(\mathbb{D}).
	\]
	this showing that the
	$\mathcal {ST}_{Ne}$-radius is    sharp.
	
	The sharpness of the
	$\mathcal {CV}_{Ne}$-radius is proved similarly. Let
	\[
	\rho
	=
	\left(
	\frac{2n}
	{M(3n+2)}
	\right)^{1/n}
	<1.
	\]
	From the definition of $\rho$, it follows $
	M\rho^n(3n+2)=2n,$
	and consequently,
	\[
	\frac{nM\rho^n}
	{n-M\rho^n}
	=
	\frac23.
	\]
	Substituting $z=z_0$ into \eqref{eq:sharp-cv}, we obtain
	\[
	1+\frac{z_0f_0''(z_0)}{f_0'(z_0)}
	=
	1-
	\frac{nM\rho^n}
	{n-M\rho^n}
	=
	\frac13
	=
	\varphi_{Ne}(-1)\in
		\partial\varphi_{Ne}(\mathbb{D}).
	\]
	which shows that the
	$\mathcal{CV}_{Ne}$-radius is   sharp.
	\end{proof}

	\begin{remark}
	For the particular class $\mathcal{Q}_1(1)$, corresponding to $n=1$ and $M=1$, the
	radius formulas in Corollary~\ref{corollary9} reduce to
	\[
	R_{\mathcal {ST}_{Ne}}(\mathcal{Q}_1(1))
	=
	\frac45
	=
	0.8,
	\quad \text{	and} \quad
	R_{\mathcal {CV}_{Ne}}(\mathcal{Q}_1(1))
	=
	\frac25
	=
	0.4.
	\]
	\end{remark}

	In 2020, Goel \textit{et al.}~\cite{MR4044913} introduced the Ma--Minda starlike class corresponding to the modified sigmoid function
		$ 	\varphi_{SG}(z)=2/(1+e^{-z})	$.
	Specifically,
		$
		\mathcal{ST}_{SG}
		=
		\mathcal{ST}(\varphi_{SG})
		$
		and
		$
		\mathcal{CV}_{SG}
		=
		\mathcal{CV}(\varphi_{SG}).
		$

		\begin{lemma}\label{lemma10}\cite[Lemma 2.2]{MR4044913}
			For $2/(1+e)<a<2e/(1+e),\ \{w:|w-a|<r_a\}\subseteq \varphi_{SG}(\mathbb{D})$, where $r_a$ is
			\[ r_a=\frac{e-1}{e+1}-\left|1-a\right|\]
		\end{lemma}
	\begin{corollary}\label{corollary10}Let
	\[
R_1
	=
	\min\left\{
	1,
	\left(
	\frac{(e-1)n(n+1)}
	{M\bigl(n(e+1)+e-1\bigr)}
	\right)^{1/n}
	\right\},
\] and \[
R_2	=
	\min\left\{
	1,
	\left(
	\frac{(e-1)n}
	{M\bigl(n(e+1)+e-1\bigr)}
	\right)^{1/n}
	\right\}.
	\]
	The $\mathcal {ST}_{SG}$-radius of the class $\Q$ is $R_1$ and the $\mathcal {CV}_{SG}$-radius  $R_2$.
	\end{corollary}
	
	\begin{proof}
 	By Lemma~\ref{lemma10}, $ \mathbb D\!\left(1;(e-1)/(e+1)\right)\subset\varphi_{SG}(\mathbb{D})$.
	Hence, Theorem~\ref{main theorem} with
	\(
	\delta=2/(1+e)
	\)
	immediately furnishes the asserted lower bounds for the
	$\mathcal {ST}_{SG}$- and $\mathcal {CV}_{SG}$-radii.
  We now show that these bounds are best possible. Choose $z_0 \in \D$ such that
		$
		|z_0|=\rho,
		 $ and $
		z_0^n=-\rho^n.
		$ Consider the extremal function $f_0:\D \rightarrow \mathbb C$ defined by
\eqref{eq:sharp-stt}.
	Consider the
	$\mathcal {ST}_{SG}$-radius. Let
	\[
	\rho
	=
	\left(
	\frac{(e-1)n(n+1)}
	{M\bigl(n(e+1)+e-1\bigr)}
	\right)^{1/n}
	<1,
	\]
	Using the definition of $\rho$, we obtain
		$
		M\rho^n\bigl(n(e+1)+e-1\bigr)
		=
		(e-1)n(n+1),
		$
	which implies
	\[
	\frac{nM\rho^n}
	{n(n+1)-M\rho^n}
	=
	\frac{e-1}{e+1}.
	\]
	Evaluating \eqref{eq:sharp-st} at $z=z_0$ gives
	\[
	\frac{z_0f_0'(z_0)}{f_0(z_0)}
	=
	1-
	\frac{nM\rho^n}
	{n(n+1)-M\rho^n}
	=
	\frac{2}{e+1}
	=
	\varphi_{SG}(-1)\in
		\partial\varphi_{SG}(\mathbb{D}).
	\]
	which proves that the
	$\mathcal {ST}_{SG}$-radius is   sharp.
	
	The proof for the
	$\mathcal {CV}_{SG}$-radius proceeds in the same way. Let
	\[
	\rho
	=
	\left(
	\frac{(e-1)n}
	{M\bigl(n(e+1)+e-1\bigr)}
	\right)^{1/n}
	<1.
	\]
	Then $
	M\rho^n\bigl(n(e+1)+e-1\bigr)
	=
	(e-1)n,$
	and therefore
	\[
	\frac{nM\rho^n}
	{n-M\rho^n}
	=
	\frac{e-1}{e+1}.
	\]
 Evaluating \eqref{eq:sharp-cv} at $z=z_0$ gives
	\[
	1+\frac{z_0f_0''(z_0)}{f_0'(z_0)}
	=
	1-
	\frac{nM\rho^n}
	{n-M\rho^n}
	=
	\frac{2}{e+1}
	=
	\varphi_{SG}(-1)\in
		\partial\varphi_{SG}(\mathbb{D}).
	\]
	which shows that the
	$\mathcal {CV}_{SG}$-radius is  sharp.
	\end{proof}

	\begin{remark}
	For the particular class $\mathcal{Q}_1(1)$, corresponding to $n=1$ and $M=1$, the
	radius formulas in Corollary~\ref{corollary10} reduce to
	\[
	R_{\mathcal {ST}_{SG}}(\mathcal{Q}_1(1))
	=
	\frac{e-1}{e}
	\approx
	0.632121,
\] and \[
 	R_{\mathcal {CV}_{SG}}(\mathcal{Q}_1(1))
	=
	\frac{e-1}{2e}
	=
	\frac12\left(1-\frac1e\right)
	\approx
	0.316060.
	\]
	\end{remark}
	
	\begin{figure}[!htbp]
		\centering
		\begin{subfigure}{0.48\textwidth}
		\centering
	\begin{tikzpicture}
		\begin{axis} [
		    axis lines=center,
		    width=8.3cm,
		    height=8.5cm,
		    xmin=-0.3, xmax=2.5,
		    ymin=-1.5, ymax=1.5,
		    xtick={0.5,1.0,1.5,2.0},
		    ytick={-1.0,-0.5,0.5,1.0},
		    thick,
		    scale=0.6
		]
		   \addplot[
		   			domain=0:360,
		   			samples=400,
		   			smooth,
		   			color=black,
		   			line width=1pt
		   		]
		   		({0.809 + 0.476*cos(x)},
		   		{0.476*sin(x)});
		
		\addplot[
					domain=0:360,
					samples=400,
					smooth,
				color=black,
					line width=1pt
				]
				({1 + 0.66667*cos(x)},
				{0.66667*sin(x)});	
		\addplot[
		    domain=-180:180,
		    samples=400,
		    smooth,
		    color=black,
		    line width=1pt,
		]
		({1 + cos(x) - (1/3) * cos(3*x)},
		 {sin(x) - (1/3) * sin(3*x)});
		
		\end{axis}
		\end{tikzpicture}
		\caption{$\mathcal{ST}_{Ne}$-Radius}
		\end{subfigure}
		\begin{subfigure}{0.48\textwidth}
		\centering
	\begin{tikzpicture}
		\begin{axis} [
		    axis lines=center,
		    width=10cm,
		    height=8.3cm,
		    xmin=-0.3, xmax=1.6,
		    ymin=-0.8, ymax=0.8,
		    xtick={0.2,0.6,1.0,1.4},
		    ytick={-0.6,-0.2,0.2,0.6},
		    thick,
		    scale=0.6
		]
		   \addplot[
		   			domain=0:360,
		   			samples=400,
		   			smooth,
		   			color=black,
		   			line width=1pt
		   		]
		   		({0.889 + 0.351*cos(x)},
		   		{0.351*sin(x)});
		
		\addplot[
					domain=0:360,
					samples=400,
					smooth,
				color=black,
					line width=1pt
				]
				({1 + 0.462117*cos(x)},
				{0.462117*sin(x)});	
		\addplot[
		    domain=-180:180,
		    samples=400,
		    smooth,
		    color=black,
		    line width=1pt,
		]
		({2 * (1 + ex(x) * cos(deg(sin(x)))) / den1(x)},
		 {2 * ex(x) * sin(deg(sin(x))) / den1(x)});
		
		\end{axis}
		\end{tikzpicture}
		\caption{$\mathcal{ST}_{SG}$-Radius}
		\end{subfigure}
		\caption{Radii associated with $\varphi_{Ne}(z)=1+z-z^3/3$ and  $\varphi_{SG}(z)={2}/(1+e^{-z})$ for the class $\Q$.}
		\label{p1f5}
			\end{figure}
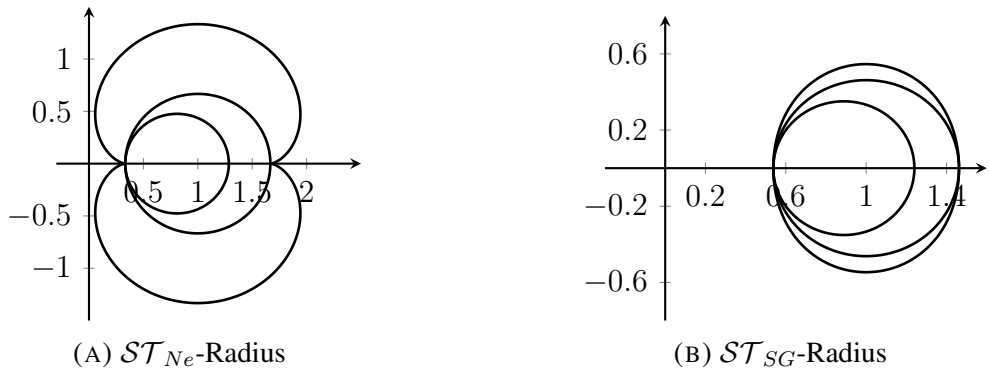
	
Gandhi~\cite{MR4099174} introduced the Ma--Minda starlike class associated with the three-leaf function
	$	\varphi_{3L}(z)=	1+4z/5+ z^{4}/5$.	Specifically
	$
	\mathcal{ST}_{3L}
	=
	\mathcal{ST}(\varphi_{3L})
	$
	and
	$
	\mathcal{CV}_{3L}
	=
	\mathcal{CV}(\varphi_{3L}),
	$
	respectively.
	\begin{lemma}\label{lemma11}\cite[Lemma 2.2]{MR4099174}
			For $2/5<a< 2,\ \{w:|w-a|<r_a\}\subseteq \varphi_{3L}(\mathbb{D})$, where $r_a$ is
			\begin{align*}
				r_a &=
				\begin{dcases}
					a-\frac{2}{5},\ & \ \frac{2}{5}<a\leq 1\\
					\sqrt{\left(a-\frac{7}{5}\right)^2+\frac{a}{5}},\ & 1\leq a\le\frac{51}{35}\\
					2-a,\ & \ \frac{51}{35}\leq a<2.
				\end{dcases}
			\end{align*}
		\end{lemma}	
	\begin{corollary}\label{corollary11}Let
	\[
R_1
	=
	\min\left\{
	1,
	\left(
	\frac{3n(n+1)}
	{M(5n+3)}
	\right)^{1/n}
	\right\},
	\quad \text{and} \quad R_2
	=
	\min\left\{
	1,
	\left(
	\frac{3n}
	{M(5n+3)}
	\right)^{1/n}
	\right\}.
	\]
	The $\mathcal {ST}_{3L}$-radius of the class $\Q$ is  $R_1$
	and the $\mathcal {CV}_{3L}$-radius$R_2$.
	\end{corollary}
	
	\begin{proof}
	From Lemma~\ref{lemma11},
	$
	\mathbb D\!\left(1;3/5\right)\subset\varphi_{3L}(\mathbb{D}).
	$
	Applying Theorem~\ref{main theorem} with
	\(
	\delta=2/5
	\)
	gives the asserted lower bounds for the
	$\mathcal{ST}_{3L} $- and $\mathcal {CV}_{3L}$-radii.
  We next verify the sharpness of these estimates.
	 Choose $z_0 \in \D$ such that
			$
			|z_0|=\rho,
			 $ and $
			z_0^n=-\rho^n.
			$ Consider the extremal function $f_0:\D \rightarrow \mathbb C$ defined by
			\eqref{eq:sharp-stt}.
	Consider first the
	$\mathcal {ST}_{3L}$-radius. Let
	\[
	\rho
	=
	\left(
	\frac{3n(n+1)}
	{M(5n+3)}
	\right)^{1/n}
	<1,
	\]
	Using the definition of $\rho$, we have $M\rho^n(5n+3)=3n(n+1),$
	which implies
	\[
	\frac{nM\rho^n}
	{n(n+1)-M\rho^n}
	=
	\frac35.
	\]
	Hence, by \eqref{eq:sharp-st}, evaluated at $z=z_0$, we have
	\[
	\frac{z_0f_0'(z_0)}{f_0(z_0)}
	=
	1-
	\frac{nM\rho^n}
	{n(n+1)-M\rho^n}
	=
	\frac25
	=
	\varphi_{3L}(-1)	\in
		\partial\varphi_{3L}(\mathbb{D}).
	\]
	which establishes the sharpness of the
	$\mathcal {ST}_{3L}$-radius.
	
	The proof for the
	$\mathcal {CV}_{3L}$-radius is similar. Let
	\[
	\rho
	=
	\left(
	\frac{3n}
	{M(5n+3)}
	\right)^{1/n}
	<1.
	\]
	Then $	M\rho^n(5n+3)=3n,$
	and therefore
	\[
	\frac{nM\rho^n}
	{n-M\rho^n}
	=
	\frac35.
	\]
	Hence, by \eqref{eq:sharp-cv}, evaluated at $z=z_0$, we have
	\[
	1+\frac{z_0f_0''(z_0)}{f_0'(z_0)}
	=
	1-
	\frac{nM\rho^n}
	{n-M\rho^n}
	=
	\frac25
	=
	\varphi_{3L}(-1)	\in
		\partial\varphi_{3L}(\mathbb{D}).
	\]
	which proves that the
	$\mathcal {CV}_{3L}$-radius is sharp.
	\end{proof}

	\begin{remark}
	For the particular class $\mathcal{Q}_1(1)$, corresponding to $n=1$ and $M=1$, the
	radius formulas in Corollary~\ref{corollary11} reduce to
	\[
	R_{\mathcal {ST}_{3L}}(\mathcal{Q}_1(1))
	=
	\frac{3}{4}
	=
	0.75,
	\quad \text{and}\quad
	R_{\mathcal {CV}_{3L}}(\mathcal{Q}_1(1))
	=
	\frac{3}{8}
	=
	0.375.
	\]
	\end{remark}	
	
	In 2022, Arora and Kumar~\cite{MR4457452} introduced the Ma--Minda starlike class  corresponding to the petal-shaped domain generated by
		$
		\displaystyle
		\varphi_{h}(z)
		=
		1+\sinh^{-1}(z).
		$
		Let
		$
		\mathcal{ST}_{h}
		=
		\mathcal{ST}(\varphi_{h})
		$
		and
		$
		\mathcal{CV}_{h}
		=
		\mathcal{CV}(\varphi_{h}).
		$

		\begin{lemma}\label{lemma12}\cite[Lemma 2.6]{MR4457452}
			For $1-\sinh^{-1}1<a<1+\sinh^{-1}1,\ \{w:|w-a|<r_a\}\subseteq \varphi_{h}(\mathbb{D})$, where $r_a$ is given by
			\begin{align*}
				r_a &=
				\begin{dcases}
					a-(1-\sinh^{-1}1),\ & \ 1-\sinh^{-1}1<a\leq1,\\
					1+\sinh^{-1}1-a,\ & 1\leq a<	1+\sinh^{-1}1.
				\end{dcases}
			\end{align*}
		\end{lemma}
	\begin{corollary}\label{corollary12}Let
	\[
R_1
	=
	\min\left\{
	1,
	\left(
	\frac{\sinh^{-1}(1)\,n(n+1)}
	{M\left(n+\sinh^{-1}(1)\right)}
	\right)^{1/n}
	\right\},
	\]and
	\[
R_2
	=
	\min\left\{
	1,
	\left(
	\frac{\sinh^{-1}(1)\,n}
	{M\left(n+\sinh^{-1}(1)\right)}
	\right)^{1/n}
	\right\}.
	\]
The $\mathcal {ST}_{h}$-radius of the class $\Q$ is $R_1$
	and the $\mathcal {}CV_{h}$-radius is $R_2$.
	\end{corollary}
	
	\begin{proof}
	By Lemma~\ref{lemma12},
	$
	\mathbb D\!\left(1;\sinh^{-1}(1)\right)\subset\varphi_h(\mathbb{D})
	$.
	Hence, applying Theorem~\ref{main theorem} with
	\(
	\delta=1-\sinh^{-1}(1)
	\)
	gives the asserted lower bounds for the
	$\mathcal {ST}_{h}$- and $\mathcal {CV}_{h}$-radii.
  We now verify that these estimates are sharp. Choose $z_0 \in \D$ satisfying
		$
		|z_0|=\rho, $ and $
		z_0^n=-\rho^n.
		$ consider the extremal function $f_0:\D \rightarrow \mathbb C$ defined by
\eqref{eq:sharp-stt}.
			We first consider the
	$\mathcal {ST}_{h}$-radius. Let
	\[
	\rho
	=
	\left(
	\frac{\sinh^{-1}(1)\,n(n+1)}
	{M\left(n+\sinh^{-1}(1)\right)}
	\right)^{1/n}
	<1,
	\]
	Using the definition of $\rho$, we obtain $M\rho^n\left(n+\sinh^{-1}(1)\right)=\sinh^{-1}(1)\,n(n+1),$
	and hence
	\[
	\frac{nM\rho^n}
	{n(n+1)-M\rho^n}
	=
	\sinh^{-1}(1).
	\]
	Evaluating \eqref{eq:sharp-st} at $z=z_0$ gives
	\[
	\frac{z_0f_0'(z_0)}{f_0(z_0)}
	=
	1-
	\frac{nM\rho^n}
	{n(n+1)-M\rho^n}
	=
	1-\sinh^{-1}(1)
	=
	\varphi_h(-1)\in
		\partial\varphi_h(\mathbb{D}).
	\]
	which establishes the sharpness of  the
	$\mathcal {ST}_{h}$-radius.
	
	The proof for the
	$\mathcal {CV}_{h}$-radius follows similarly. Let
	\[
	\rho
	=
	\left(
	\frac{\sinh^{-1}(1)\,n}
	{M\left(n+\sinh^{-1}(1)\right)}
	\right)^{1/n}
	<1.
	\]
	Then $
	M\rho^n\left(n+\sinh^{-1}(1)\right)
	=
	\sinh^{-1}(1)\,n$,
	which gives
	\[
	\frac{nM\rho^n}
	{n-M\rho^n}
	=
	\sinh^{-1}(1).
	\]
	Evaluating \eqref{eq:sharp-cv} at $z=z_0$ gives
	\[
	1+\frac{z_0f_0''(z_0)}{f_0'(z_0)}
	=
	1-
	\frac{nM\rho^n}
	{n-M\rho^n}
	=
	1-\sinh^{-1}(1)
	=
	\varphi_h(-1)\in
		\partial\varphi_h(\mathbb{D}).
	\]
	which proves that the
	$\mathcal {CV}_{h}$-radius is sharp.
	\end{proof}

	\begin{remark}
	For the particular class $\mathcal{Q}_1(1)$, corresponding to $n=1$ and $M=1$, the
	radius formulas in Corollary~\ref{corollary12} reduce to
	\[
	R_{\mathcal {ST}_h}(\mathcal{Q}_1(1))
	=
	\frac{2\sinh^{-1}(1)}
	{1+\sinh^{-1}(1)}
	\approx
	0.936947,
\] and \[
 	R_{\mathcal {CV}_h}(\mathcal{Q}_1(1))
	=
	\frac{\sinh^{-1}(1)}
	{1+\sinh^{-1}(1)}
	\approx
	0.468473.
	\]
	\end{remark}

	\begin{figure}[!htbp]
		\centering
		\begin{subfigure}{0.48\textwidth}
		\centering
	\begin{tikzpicture}
		\begin{axis} [
		    axis lines=center,
		    width=8.8cm,
		    height=8.5cm,
		    xmin=-0.3, xmax=2.2,
		    ymin=-1.2, ymax=1.2,
		    xtick={0.5,1.0,1.5,2.0},
		    ytick={-1.0,-0.5,0.5,1.0},
		    thick,
		    scale=0.7
		]
		   \addplot[
		   			domain=0:360,
		   			samples=400,
		   			smooth,
		   			color=black,
		   			line width=1pt
		   		]
		   		({0.836 + 0.436*cos(x)},
		   		{0.436*sin(x)});
		
		\addplot[
					domain=0:360,
					samples=400,
					smooth,
				color=black,
					line width=1pt
				]
				({1 + 0.6*cos(x)},
				{0.6*sin(x)});	
		\addplot[
		    domain=-180:180,
		    samples=400,
		    smooth,
		    color=black,
		    line width=1pt,
		]
		({1 + 0.8 * cos(x) + 0.2 * cos(4*x)},
		 {0.8 * sin(x) + 0.2 * sin(4*x)});
		
		\end{axis}
		\end{tikzpicture}
		\caption{$\mathcal{ST}_{3L}$-Radius}
		\end{subfigure}
		\begin{subfigure}{0.48\textwidth}
		\centering
		\begin{tikzpicture}
		\begin{axis} [
		    axis lines=center,
		    width=9cm,
		    height=9.5cm,
		    xmin=-0.3, xmax=2.5,
		    ymin=-1.5, ymax=1.5,
		    xtick={0.5,1.0,1.5,2.0},
		    ytick={-1.0,-0.5,0.5,1.0},
		    thick,
		    scale=0.6
		]
		   \addplot[
		   			domain=0:360,
		   			samples=400,
		   			smooth,
		   			color=black,
		   			line width=1pt
		   		]
		   		({0.718 + 0.6*cos(x)},
		   		{0.6*sin(x)});
		
		\addplot[
					domain=0:360,
					samples=400,
					smooth,
				color=black,
					line width=1pt
				]
				({1 + 0.881374*cos(x)},
				{0.881374*sin(x)});	
		\addplot[
		    domain=-3.14159265:3.14159265,
		    samples=300,
		    smooth,
		    color=black,
		    line width=1pt,
		]
		({qrf(x)}, {qif(x)});	
		\end{axis}
		\end{tikzpicture}
		\caption{$\mathcal{ST}_{h}$-Radius}
		\end{subfigure}
		\caption{Radii associated with $\varphi_{3L}(z)=1+(4/5)z+(1/5)z^4 $ and  $\varphi_h(z)=1+\sinh^{-1}z$ for the class $\Q$.}
		\label{p1f6}
			\end{figure}
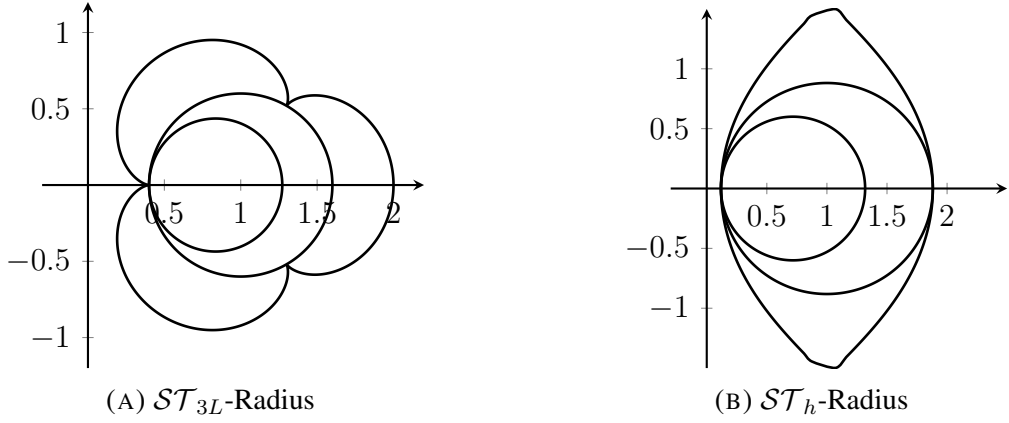
	

	In 2023, Mundalia and Kumar~\cite{MR4620985} introduced a family of Ma--Minda starlike functions  associated with the hyperbolic cosine function.  Let
		$	\mathcal{ST}_{\varrho}
		=
		\mathcal{ST}(\varphi_{\varrho})
		$
		and
		$
		\mathcal{CV}_{\varrho}
		=
		\mathcal{CV}(\varphi_{\varrho})
		$ where
		$	\varphi_{\varrho}(z)=\cosh\sqrt{z}$.  Our results rely on Lemma~2.1 of~\cite{MR4620985}.
	
	\begin{lemma}\label{lemma13}\cite[Lemma 2.1]{MR4620985}
			For $ \cos1<a<\cosh1,\ \{w:|w-a|<r_a\}\subseteq \varphi_{\varrho }(\mathbb{D})$, where $r_a$ is
			\begin{align*}
				r_a &=
				\begin{dcases}
					a-\cos1,\ & \ \cos1<a\leq\frac{\cos1+\cosh1}{2} \\
					\cosh1-a,\ & \frac{\cos1+\cosh1}{2}\leq a<\cosh1\\
				\end{dcases}
			\end{align*}
		\end{lemma}	
	\begin{corollary}\label{corollary13}
Let
	\[
	R_1
	=
	\min\left\{
	1,
	\left(
	\frac{(1-\cos1)n(n+1)}
	{M(n+1-\cos1)}
	\right)^{1/n}
	\right\},
	\]
	and
	\[
	R_2
	=
	\min\left\{
	1,
	\left(
	\frac{(1-\cos1)n}
	{M(n+1-\cos1)}
	\right)^{1/n}
	\right\}.
	\]
	The $\mathcal {ST}_{\varrho}$-radius of the class $\Q$ is $R_1$ and  the $\mathcal {CV}_{\varrho}$-radius is $R_2$.
	\end{corollary}
	
	\begin{proof}
	From Lemma~\ref{lemma13},
	$
	\mathbb D\!\left(1;1-\cos1\right)\subset\varphi_{\varrho}(\mathbb{D}).
	$
	Applying Theorem~\ref{main theorem} with
	\(
	\delta=\cos1
	\)
	gives the asserted lower bounds for the
	$\mathcal {ST}_{\varrho}$- and $\mathcal {CV}_{\varrho}$-radii.
		We next establish the sharpness of these estimates. Choose $z_0 \in \D$ satisfying
			$
			|z_0|=\rho, $ and $
			z_0^n=-\rho^n.
			$ consider the extremal function $f_0:\D \rightarrow \mathbb C$ defined by
\eqref{eq:sharp-stt}.  Consider first the
	$\mathcal {ST}_{\varrho}$-radius. Let
	\[
	\rho
	=
	\left(
	\frac{(1-\cos1)n(n+1)}
	{M(n+1-\cos1)}
	\right)^{1/n}
	<1,
	\]
	Using the definition of $\rho$, we obtain
	$
	M\rho^n(n+1-\cos1)
	=
	(1-\cos1)n(n+1),
	$
	which implies
	\[
	\frac{nM\rho^n}
	{n(n+1)-M\rho^n}
	=
	1-\cos1.
	\]
	Hence, by \eqref{eq:sharp-st}, evaluated at $z=z_0$, we have
	\[
	\frac{z_0f_0'(z_0)}{f_0(z_0)}
	=
	1-
	\frac{nM\rho^n}
	{n(n+1)-M\rho^n}
	=
	\cos1
	=
	\varphi_{\varrho}(-1)\in
		\partial\varphi_{\varrho}(\mathbb{D}).
	\]
	which establishes the sharpness of the
	$\mathcal {ST}_{\varrho}$-radius.
	
	The proof for the
	$\mathcal {CV}_{\varrho}$-radius is analogous. Let
	\[
	\rho
	=
	\left(
	\frac{(1-\cos1)n}
	{M(n+1-\cos1)}
	\right)^{1/n}
	<1.
	\]
	Again,$
	M\rho^n(n+1-\cos1)
	=
	(1-\cos1)n,
	$
	and hence
	\[
	\frac{nM\rho^n}
	{n-M\rho^n}
	=
	1-\cos1.
	\]
	Hence, by \eqref{eq:sharp-cv}, evaluated at $z=z_0$, we have
	\[
	1+\frac{z_0f_0''(z_0)}{f_0'(z_0)}
	=
	1-
	\frac{nM\rho^n}
	{n-M\rho^n}
	=
	\cos1
	=
	\varphi_{\varrho}(-1)\in
		\partial\varphi_{\varrho}(\mathbb{D}).
	\]
	which proves that the
	$\mathcal {CV}_{\varrho}$-radius is  sharp.
	\end{proof}
	
	\begin{remark}
	For the particular class $\mathcal{Q}_1(1)$, corresponding to $n=1$ and $M=1$, the
	radius formulas in Corollary~\ref{corollary13} reduce to
	\[
	R_{\mathcal {ST}_{\varrho}}(\mathcal{Q}_1(1))
	=
	\frac{2(1-\cos1)}
	{2-\cos1}
	\approx
	0.629853,
\] and \[
	R_{\mathcal {CV}_{\varrho}}(\mathcal{Q}_1(1))
	=
	\frac{1-\cos1}
	{2-\cos1}
	\approx
	0.314927.
	\]
	\end{remark}

\section{Lower Bounds}

The method developed in Theorem~\ref{main theorem} applies to Ma--Minda
domains for which the largest disk centered at $1$ is of the form
$
\mathbb D(1,1-\delta),
$
where $\delta$ denotes the point at which the boundary of the domain
intersects the real axis to the left of $1$. In this situation, the
extremal condition is determined by the left endpoint of the disk.
The Bernoulli lemniscate domain
$
\varphi_L(\mathbb{D})=\{w\in\mathbb C:|w^2-1|<1\}
$
is exceptional, since its largest disk centered at $1$ has
radius $\sqrt2-1$. Hence the extremal condition is determined by the right
endpoint of the disk rather than the left endpoint, so the proof of
Theorem~\ref{main theorem} is not applicable and a separate argument is
required.

In 1996, Sok\'o{\l} and Stankiewicz~\cite{MR1473947} introduced the Ma--Minda classes associated with the right-half of the Bernoulli lemniscate (see Khatter \textit{et al.}~\cite{MR3902942} for a generalization). These classes correspond to the function
$\varphi_{L}(z)=\sqrt{1+z}$.
Specifically, $\mathcal{ST}_{L}=\mathcal{ST}(\varphi_{L})$ and $\mathcal{CV}_{L}=\mathcal{CV}(\varphi_{L})$. We shall also use the inclusion result established in \cite[Lemma~2.3]{MR3902942}.

The following theorem establishes lower bounds for the
$\mathcal {ST}_L$-radius and $\mathcal {CV}_L$-radius of the class $\Q$.
	
	\begin{lemma}\label{lemma14}\cite{MR3902942,MR4314116}\label{limlemma}
		For $0<a<\sqrt{2},\ \{w:|w-a|<r_a\}\subseteq \varphi_{L}(\mathbb{D})$, where $r_a$ is
		\begin{align*}
			r_a &=
			\begin{dcases}
				\sqrt{\sqrt{(1-a)(1+a)}-(1-a)(1+a)},\ & \ 0<a\leq\frac{2\sqrt{2}}{3},\\
				\sqrt{2}-a,\ & \frac{2\sqrt{2}}{3}\leq  a<\sqrt{2}.
			\end{dcases}
		\end{align*}
	\end{lemma}

	\begin{figure}[htb!]
		\centering
		\begin{subfigure}{0.48\textwidth}
		\centering
		
	\begin{tikzpicture}
		\begin{axis} [
		    axis lines=center,
		    width=11.3cm,
		    height=8.5cm,
		    xmin=-0.3, xmax=1.75,
		    ymin=-0.75, ymax=0.75,
		    xtick={0.5,1.0,1.5},
		    ytick={-0.5,0.5},
		    thick,
		    scale=0.55
		]
		   \addplot[
		   			domain=0:360,
		   			samples=400,
		   			smooth,
		   			color=black,
		   			line width=1pt
		   		]
		   		({0.889 + 0.349*cos(x)},
		   		{0.349*sin(x)});
		
		\addplot[
					domain=0:360,
					samples=400,
					smooth,
				color=black,
					line width=1pt
				]
				({1 + 0.4596*cos(x)},
				{0.4596*sin(x)});	
		\addplot[
		    domain=-180:180,
		    samples=400,
		    smooth,
		    color=black,
		    line width=1pt,
		]
		({cosh(cos(x)) * cos(deg(sin(x)))},
		 {sinh(cos(x)) * sin(deg(sin(x)))});
		
		\end{axis}
		\end{tikzpicture}
	
		\caption{$\mathcal{ST}_{\varrho}$-Radius}
		\end{subfigure}
		\begin{subfigure}{0.48\textwidth}
		\centering
		
	\begin{tikzpicture}
		\begin{axis} [
		    axis lines=center,
		    width=10.5cm,
		    height=9cm,
		    xmin=-0.3, xmax=1.6,
		    ymin=-0.8, ymax=0.8,
		    xtick={0.5,1,1.4},
		    ytick={-0.2,-0.6,0.2,0.6,1.0},
		    thick,
		    scale=0.5
		]
		   \addplot[
		   			domain=0:360,
		   			samples=400,
		   			smooth,
		   			color=black,
		   			line width=1pt
		   		]
		   		({0.906+ 0.320*cos(x)},
		   		{0.320*sin(x)});
		
		\addplot[
					domain=0:360,
					samples=400,
					smooth,
				color=black,
					line width=1pt
				]
				({1 + 0.414214*cos(x)},
				{0.414214*sin(x)});	
		\addplot[
		    domain=-180:180,
		    samples=400,
		    smooth,
		    color=black,
		    line width=1pt,
		]
		({sqrt(2*cos(x/2)) * cos(x/4)},
		 {sqrt(2*cos(x/2)) * sin(x/4)});
		
		\end{axis}
		\end{tikzpicture}
		\caption{$\mathcal{ST}_{L}$-Radius}
		\end{subfigure}
		\caption{Radii associated with $\varphi_{\varrho}(z)=\cosh(\sqrt z)$ and $\varphi_L(z)=\sqrt{1+z}$ for the class $\Q$.}
		\label{p1f7}
			\end{figure}
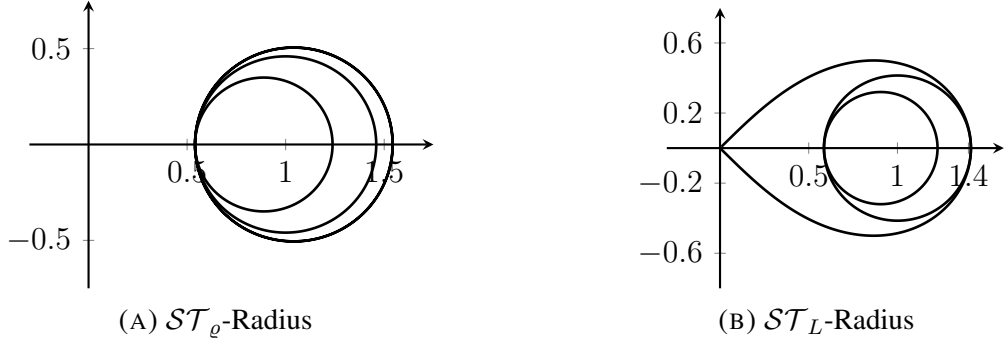

	\begin{theorem}\label{thm:lemniscate}
	Let $\varphi_{L}(z)=\sqrt{1+z}$,
	$
	\mathbb D(1;\sqrt2-1)\subseteq \varphi_L(\mathbb{D})
	$,
	Let \[
	R_1=
	\min\left\{
		1,
		\left(
		\frac{(\sqrt2-1)n(n+1)}
		{M(n+\sqrt2-1)}
		\right)^{1/n}
		\right\},
	\] and \[
	R_2=\min\left\{
		1,
		\left(
		\frac{(\sqrt2-1)n}
		{M(n+\sqrt2-1)}
		\right)^{1/n}
		\right\}.
	\]
	The $\mathcal{ST}_L$-radius of the class $\Q$ is at least $R_1$ whenever $0<M\le n(n+1)$, whereas the $\mathcal{CV}_L$-radius is at least $R_2$ whenever $0<M\le n$.
	\end{theorem}	
	
	\begin{proof}
	Let $f\in\Q$ with $0<M\le n(n+1)$. By Lemma~\ref{mainlemma},
	$
	{zf'(z)}/{f(z)}
	\in
	\D(1;c_1(r)),
	$
	where
	\[
	c_1(r)=\frac{nMr^n}{n(n+1)-Mr^n}.
	\]
	By Lemma~\ref{limlemma},
		$
		\D(1,\sqrt2-1)\subseteq\varphi_L(\mathbb{D}).
		$
	it suffices to determine the largest radius $r$ for which
	$
	\D(1;c_1(r))
	\subseteq
	\D(1;\sqrt2-1)
	$. Let
	\[
	h(r)
	=
	1+c_1(r)
	=
	1+\frac{nMr^n}{n(n+1)-Mr^n}.
	\]
	$h(r)$ denotes the right endpoint on the real axis of the disk $\D(1;c_1(r))$.
	A direct computation gives
	\[
	h'(r)
	=
	\frac{Mn^3(n+1)r^{\,n-1}}
	{\left(n(n+1)-Mr^n\right)^2}
	>0,
	\qquad
	0<r<1.
	\]
	Hence $h$ is strictly increasing on $(0,1)$, meaning that the disk
	$\D(1;c_1(r))$ expands and its right endpoint moves monotonically to the right
	as $r$ increases. The largest admissible radius is determined when $h(r)$ reaches the
	rightmost boundary point $\sqrt2$ Thus we solve
	$
	h(r)=\sqrt2,
	$
	or equivalently,
	$
	c_1(r)=\sqrt2-1.
	$
	This reduces to
	\[
	\Phi(r)
	:=
	Mr^n(n+\sqrt2-1)
	-
	(\sqrt2-1)n(n+1).
	\]
	Whether this equation admits a solution in $(0,1)$ depends on the value of
	$M$, and hence we consider two cases.
	
	\medskip
	
	\noindent\textbf{Case (i).}
	Suppose that
	\[
	\frac{(\sqrt2-1)n(n+1)}
	{n+\sqrt2-1}
	<
	M
	\le
	n(n+1).
	\]
	Then
	\[
	\Phi(0)
	=
	-(\sqrt2-1)n(n+1)
	<0,
	\]
	whereas
	\[
	\Phi(1)
	=
	M(n+\sqrt2-1)
	-
	(\sqrt2-1)n(n+1)
	>0.
	\]
	Since $\Phi$ is continuous and strictly increasing on $(0,1)$, it has a unique
	zero in $(0,1)$, namely
	\[
	R_1
	=
	\left(
	\frac{(\sqrt2-1)n(n+1)}
	{M(n+\sqrt2-1)}
	\right)^{1/n}.
	\]
	
	\medskip
	
	\noindent\textbf{Case (ii).}
	Suppose that
	\[
	0<M
	\le
	\frac{(\sqrt2-1)n(n+1)}
	{n+\sqrt2-1}.
	\]
	Then
	\[
	\Phi(1)
	=
	M(n+\sqrt2-1)
	-
	(\sqrt2-1)n(n+1)
	\le0.
	\]
	Since $\Phi$ is increasing on $(0,1)$,
	$
	\Phi(r)\le0,
	$ for all $
	0<r<1.
	$
	Equivalently,
	\[
	c_1(r)
	=
	\frac{nMr^n}
	{n(n+1)-Mr^n}
	\le
	\sqrt2-1.
	\]
	Hence
	$
	\D(1;c_1(r))
	\subseteq
	\D(1;\sqrt2-1),
	$ for $
	0<r<1,
	$
	and therefore
	$
	R_1=1.
	$
	Combining the two cases,
	\[
	R_1
	=
	\min\left\{
	1,
	\left(
	\frac{(\sqrt2-1)n(n+1)}
	{M(n+\sqrt2-1)}
	\right)^{1/n}
	\right\}.
	\]
	Finally, since $h$ is increasing,
	$
	h(r)\le h(R_1)=\sqrt2,
	$ for all $
	0<r\le R_1.
	$
	Hence
	$
	c_1(r)\le\sqrt2-1,
	$
	so that
	$
	\D(1;c_1(r))
	\subseteq
	\D(1;\sqrt2-1)
	\subseteq
	\varphi_L(\mathbb{D}),
	0\le r\le R_1.
	$
	Therefore, $
	{zf'(z)}/{f(z)}
	\in
	\varphi_L(\mathbb{D}),
	$ for $
	|z|\le R_1,
	$
	which proves
	$
	R_{\mathcal{ST}_{L}}(\Q)
	\ge
	R_1
	$.
	
	Next, let $f\in\Q$ with $0<M\le n$. By Lemma~\ref{mainlemma}, we have
	$
	1+{zf''(z)}/{f'(z)}
	\in
	\D(1;c_2(r)),
	$
	where
	\[
	c_2(r)=\frac{nMr^n}{n-Mr^n}.
	\]
	By Lemma~\ref{limlemma},
	$
	\D(1,\sqrt2-1)\subseteq\varphi_L(\mathbb{D}).
	$
	It suffices to determine the largest radius $r$ for which
	$
	\D(1;c_2(r))
	\subseteq
	\D(1;\sqrt2-1).
	$
	Let
	\[
	g(r)
	=
	1+c_2(r)
	=
	1+\frac{nMr^n}{n-Mr^n}.
	\]
	Thus, $g(r)$ denotes the right endpoint of the disk estimate on the real axis.
	A direct computation gives
	\[
	g'(r)
	=
	\frac{Mn^3r^{\,n-1}}
	{\left(n-Mr^n\right)^2}
	>0,
	\qquad
	0<r<1.
	\]
	Hence $g$ is strictly increasing on $(0,1)$, meaning that the disk
	$\D(1;c_2(r))$ expands and its right endpoint moves monotonically to the right
	as $r$ increases.
	The largest admissible radius is determined when $g(r)$ reaches the
	rightmost boundary point $\sqrt2$.
	Thus we solve
	$
	g(r)=\sqrt2,
	$
	or equivalently,
	$
	c_2(r)=\sqrt2-1.
	$
	This reduces to
	\[
	\Psi(r)
	:=
	Mr^n(n+\sqrt2-1)
	-
	(\sqrt2-1)n.
	\]
	The location of the solution depends on $M$, leading to two cases.
	
	\medskip
	
	\noindent\textbf{Case (i).}
	Suppose that
	\[
	\frac{(\sqrt2-1)n}
	{n+\sqrt2-1}
	<
	M
	\le
	n.
	\]
	Then
	\[
	\Psi(0)
	=
	-(\sqrt2-1)n
	<0,
	\]
	whereas
	\[
	\Psi(1)
	=
	M(n+\sqrt2-1)
	-
	(\sqrt2-1)n
	>0.
	\]
	Since $\Psi$ is continuous and strictly increasing on $(0,1)$, it has a unique
	zero in $(0,1)$, namely
	\[
	R_2
	=
	\left(
	\frac{(\sqrt2-1)n}
	{M(n+\sqrt2-1)}
	\right)^{1/n}.
	\]
	
	\medskip
	
	\noindent\textbf{Case (ii).}
	Suppose that
	\[
	0<M
	\le
	\frac{(\sqrt2-1)n}
	{n+\sqrt2-1}.
	\]
	Then
	\[
	\Psi(1)
	=
	M(n+\sqrt2-1)
	-
	(\sqrt2-1)n
	\le0.
	\]
	Since $\Psi$ is increasing on $(0,1)$,
	$
	\Psi(r)\le0,
		$ for all $
	0<r<1.
	$
	Equivalently,
	\[
	c_2(r)
	=
	\frac{nMr^n}
	{n-Mr^n}
	\le
	\sqrt2-1.
	\]
	Hence
	$
	\D(1;c_2(r))
	\subseteq
	\D(1;\sqrt2-1),
	$ for all $
	0<r<1,
	$
	and therefore
	$
	R_2=1.
	$
 	Combining the two cases,
	\[
	R_2
	=
	\min\left\{
	1,
	\left(
	\frac{(\sqrt2-1)n}
	{M(n+\sqrt2-1)}
	\right)^{1/n}
	\right\}.
	\]
	Finally, since $g$ is increasing,
	$
	g(r)\le g(R_2)=\sqrt2,
	$ for all $
	0<r\le R_2.
	$
	Hence
	$
	c_2(r)\le\sqrt2-1,
	$
	so that
	$
	\D(1;c_2(r))
	\subseteq
	\D(1;\sqrt2-1)
	\subseteq
	\varphi_L(\mathbb{D}),
	0\le r\le R_2.
	$
	Therefore,
	$
	1+{zf''(z)}/{f'(z)}
	\in
	\varphi_L(\mathbb{D}),
	$ for $
	|z|\le R_2,
	$
	which proves
	$
	R_{\mathcal {CV}_{L}}(\Q)
	\ge
	R_2
	$.
	\end{proof}
		\begin{remark}
		For the particular class $\mathcal{Q}_1(1)$, corresponding to $n=1$ and $M=1$, the
		radius formulas in Theorem~\ref{thm:lemniscate} reduce to
		\[
		R_{\mathcal {ST}_{L}}(\mathcal{Q}_1(1))
		= \sqrt2(\sqrt2-1)
		\approx
		0.585786,
		\]
		and
		\[
		R_{\mathcal {CV}_{L}}(\mathcal{Q}_1(1))
		=1-\frac{1}{\sqrt2}
		\approx
		0.292893.
		\]
		\end{remark}
	The obtained radius estimate is only a  lower bound.

\bigskip
%
%

\noindent
\textbf{Acknowledgement}: The first author is supported by a HTRA fellowship from NIT Tiruchirappalli.


\begin{thebibliography}{00}
	\bibitem{MR4457452}
	K. Arora and S.~S. Kumar, Starlike functions associated with a petal shaped domain, Bull. Korean Math. Soc. {\bf 59} (2022), no.~4, 993--1010.
	
	\bibitem{MR3913990}
	N.~E. Cho, V. Kumar, S. S. Kumar, and V. Ravichandran, Radius problems for starlike functions associated with the sine function, Bull. Iranian Math. Soc. {\bf 45} (2019), no.~1, 213--232.
	
	\bibitem{MR4386722}
	N.~E. Cho, A. Swaminathan and L.~A. Wani, Radius constants for functions associated with a limacon domain, J. Korean Math. Soc. {\bf 59} (2022), no.~2, 353--365.
	
	\bibitem{MR3718233}
	S. Gandhi and V. Ravichandran, Starlike functions associated with a lune, Asian-Eur. J. Math. {\bf 10} (2017), no.~4, 1750064, 12 pp.
	
	\bibitem{MR4099174}
	S. Gandhi, Radius estimates for three leaf function and convex combination of starlike functions, in {\it Mathematical analysis. I. Approximation theory}, 173--184, Springer Proc. Math. Stat., 306, Springer, Singapore.
	
	\bibitem{MR4044913}
	P. Goel and S.~S. Kumar, Certain class of starlike functions associated with modified sigmoid function, Bull. Malays. Math. Sci. Soc. {\bf 43} (2020), no.~1, 957--991.
	
	\bibitem{MR4303808}
	P. Gupta, S. Nagpal and V. Ravichandran, Inclusion relations and radius problems for a subclass of starlike functions, J. Korean Math. Soc. {\bf 58} (2021), no.~5, 1147--1180.
	
	\bibitem{MR4478088}
	R. Kanaga and V. Ravichandran, Radius of Lima\c con starlikeness for Janowski starlike functions, Asian-Eur. J. Math. {\bf 15} (2022), no.~9, Paper No. 2250164, 13 pp.
	
	\bibitem{MR3902942}
	K. Khatter, V. Ravichandran and S.~S. Kumar, Starlike functions associated with exponential function and the lemniscate of Bernoulli, Rev. R. Acad. Cienc. Exactas F\'is. Nat. Ser. A Mat. RACSAM {\bf 113} (2019), no.~1, 233--253.	
	
	\bibitem{mainpaper}
	U. Korzybska, K. Kowalczyk and P. W\'ojcik, On a convex class of univalent functions, Ann. Univ. Mariae Curie-Sk\l odowska Sect. A {\bf 60} (2006), 31--38.

	\bibitem{MR4216363}
	S.~S. Kumar and G. Kamaljeet, A cardioid domain and starlike functions, Anal. Math. Phys. {\bf 11} (2021), no.~2, Paper No. 54, 34 pp. 	

	\bibitem{MR3496681}
	S. Kumar and V. Ravichandran, A subclass of starlike functions associated with a rational function, Southeast Asian Bull. Math. {\bf 40} (2016), no.~2, 199--212.
	
	\bibitem{MR4314116}
	A. Lecko, V. Ravichandran and A. Sebastian, Starlikeness of certain non-univalent functions, Anal. Math. Phys. {\bf 11} (2021), no.~4, Paper No. 163, 23 pp.
	
		
	\bibitem{Masih}
	V. S. Masih,  and S. Kanas, Subclasses of starlike and convex functions associated with the limaçon domain,  Symmetry \textbf{12} (2020), no,~6,  Art. 942.	
	
	\bibitem{MR3394060}
	R. Mendiratta, S. Nagpal and V. Ravichandran, On a subclass of strongly starlike functions associated with exponential function, Bull. Malays. Math. Sci. Soc. {\bf 38} (2015), no.~1, 365--386.

	\bibitem{MR4620985}
	M. Mundalia and S.~S. Kumar, On a subfamily of starlike functions related to hyperbolic cosine function, J. Anal. {\bf 31} (2023), no.~3, 2043--2062.

	\bibitem{MR3419845}
	R.~K. Raina and J. Sok\'o\l, Some properties related to a certain class of starlike functions, C. R. Math. Acad. Sci. Paris {\bf 353} (2015), no.~11, 973--978.
	
	\bibitem{MR3536076}
	K. Sharma, N.~K. Jain and V. Ravichandran, Starlike functions associated with a cardioid, Afr. Mat. {\bf 27} (2016), no.~5-6, 923--939.
	
	\bibitem{MR1473947}
	J. Sok\'o\l{} and J. Stankiewicz, Radius of convexity of some subclasses of strongly starlike functions, Zeszyty Nauk. Politech. Rzeszowskiej Mat. No. 19 (1996), 101--105.

	\bibitem{MR4190740}
	L.~A. Wani and A. Swaminathan, Starlike and convex functions associated with a nephroid domain, Bull. Malays. Math. Sci. Soc. {\bf 44} (2021), no.~1, 79--104.
	

	\end{thebibliography}
\end{document}